\documentclass[10pt,a4paper]{article}

\usepackage{amsmath}
\usepackage{amssymb}
\usepackage{amsthm}
\usepackage{amscd}

\usepackage{graphicx,color}     
\usepackage{epsfig}

\newtheoremstyle{plain2}{\topsep}{\topsep}%
     {\itshape}
     {}
     {\bfseries}
     {.}
     {.5em}
     {\thmnumber{(#2)}\thmname{ #1}\thmnote{ #3}}

\theoremstyle{plain2}
\newtheorem{theorem}{Theorem}[section]
\newtheorem{proposition}[theorem]{Proposition}
\newtheorem{corollary}[theorem]{Corollary}
\newtheorem{lemma}[theorem]{Lemma}

\newtheoremstyle{definition2}{\topsep}{\topsep}%
     {}
     {}
     {\bfseries}
     {.}
     {.5em}
     {\thmnumber{(#2)}\thmname{ #1}\thmnote{ #3}}

\theoremstyle{definition2}

\newtheorem{remark}[theorem]{Remark}
\newtheorem{definition}[theorem]{Definition}

\def\R{\mathbb{R}}
\def\C{\mathbb{C}}

\def\spann{\mathrm{span}}
\def\Ker{\mathrm{Ker}}

\def\a{\alpha}
\def\m{\mu}
\def\l{\lambda}
\def\Om{\Omega}
\def\eps{\varepsilon}
\def\vphi{\varphi}
\def\spc{H^1_0(\Omega)}

\def\Kwong{Z_{N,p}}
\def\gslev{{{e}}}
\def\gsset{{\mathcal{P}}}
\def\Lcal{{\mathcal{L}}}

\title{Existence and orbital stability of the ground states with prescribed mass for the $L^2$-critical and supercritical NLS on bounded domains
}
\author{Benedetta Noris, Hugo Tavares and Gianmaria Verzini}

\begin{document}

\maketitle

\begin{abstract}
Given $\rho>0$, we study the elliptic problem
\[
\text{find } (U,\lambda)\in H^1_0(B_1)\times \R
\text{ such that }
\begin{cases}
-\Delta U+\lambda U=U^p\\
\int_{B_1} U^2\, dx=\rho,\quad U>0,
\end{cases}
\]
where $B_1\subset\R^N$ is the unitary ball and $p$ is Sobolev-subcritical.
Such problem arises in the search of solitary wave solutions for nonlinear Schr\"odinger equations
(NLS) with power nonlinearity on bounded domains. Necessary and sufficient conditions (about
$\rho$, $N$ and $p$) are provided for the existence of solutions. Moreover, we show that standing
waves associated to least energy solutions are orbitally stable for every $\rho$ (in the existence range) when $p$ is
$L^2$-critical and subcritical, i.e. $1<p\leq1+4/N$,  while they are stable for almost every $\rho$
in the $L^2$-supercritical regime $1+4/N<p<2^*-1$. The proofs
are obtained in connection with the study of a variational problem with two constraints, of
independent interest: to maximize the $L^{p+1}$-norm among functions having prescribed $L^2$ and
$H^1_0$-norm.
\end{abstract}

\section{Introduction}

In this paper we study standing wave solutions of the following nonlinear Schr\"o\-dinger equation (NLS)
\begin{equation}\label{eq:Schrodinger_intro}
\left\{
\begin{array}{ll}
i\frac{\partial \Phi}{\partial t}+\Delta \Phi+ |\Phi|^{p-1}\Phi=0  &
(t,x)\in \R\times B_1\\
\Phi(t,x)=0 &(t,x)\in \R\times \partial B_1,
\end{array}
\right.
\end{equation}
$B_1$ being the unitary ball of $\R^N$, $N\geq1$, and $1<p<2^*-1$, where $2^*=\infty$ if $N=1,2$,
and $2^*=2N/(N-2)$ otherwise. In what follows, $p$ is always subcritical for the Sobolev immersion,
while criticality will be understood in the $L^2$-sense, see below.

NLS on bounded domains appear in different physical contexts. For instance, in nonlinear optics,
with $N=2$, $p=3$ they describe the propagation of laser beams in hollow-core fibers
\cite{agrawal2000, FibichMerle2001}. In Bose-Einstein condensation, when $N\leq 3$ and $p=3$, they
model the presence of an infinite well trapping potential \cite{BartschParnet}. When considered in
the whole space $\R^N$, this equation admits the $L^2$-critical exponent $p=1+4/N$; indeed, in the
subcritical case $1<p<1+4/N$, ground state solutions are orbitally stable, while in the critical and
supercritical one they are always unstable \cite{CazenaveLions1982,Cazenave2003}. Notice that the
exponent $p=3$ is subcritical when $N=1$, critical when $N=2$ and supercritical when $N=3$. In the
case of a bounded domain, few papers analyze the effect of  boundary conditions on stability, namely
the one by Fibich and Merle \cite{FibichMerle2001}, and the more recent \cite{Fukuizumi2012} by
Fukuizumi, Selem and Kikuchi. In these papers, it is proved that also in the critical and
supercritical cases there exist standing waves which are orbitally stable (even though a full
classification is not provided, even in the subcritical range). This shows that the presence of the
boundary has a stabilizing effect.

As it is well known, two quantities are conserved along trajectories of \eqref{eq:Schrodinger_intro}: the energy
\[
\mathcal{E}(\Phi)=\int_{B_1} \left(\frac{1}{2}|\nabla \Phi|^2-\frac{1}{p+1}|\Phi|^{p+1}\right)\, dx
\]
and the mass
\[
\mathcal{Q}(\Phi)=\int_{B_1} |\Phi|^2\, dx.
\]
A standing wave is a solution of the form $\Phi(t,x)=e^{i\lambda t}U(x)$, where the real valued function $U$ solves the elliptic problem
\begin{equation}\label{eq:stationary_intro}
\left\{
\begin{array}{ll}
-\Delta U+\lambda U=|U|^{p-1}U& \text{ in }  B_1\\
U=0 &\text{ on } \partial B_1.
\end{array}
\right.
\end{equation}
In \eqref{eq:stationary_intro}, one can either consider the chemical potential $\lambda\in\R$ to be given, or to be an unknown of the problem. In the latter case, it is natural to prescribe the value of the mass, so that $\lambda$ can be interpreted as a Lagrange multiplier.

Among all possible standing waves, typically the most relevant are ground state solutions. In the literature, the two points of view mentioned above lead to different definitions of ground state, see for instance \cite{AdamiNojaVisciglia2012}. When $\lambda$ is prescribed, ground states can be defined as minimizers of the action functional
\[
\mathcal{A}_\lambda(U)=\mathcal{E}(U)+\lambda \mathcal{Q}(U)
\]
among its nontrivial critical points (recall that $\mathcal{A}_\lambda$ is not bounded from below), see for instance \cite[p. 316]{BerestyckiLions1983}. Equivalently, they can be defined as minimizers of $\mathcal{A}_\lambda$ on the associated Nehari manifold. Even though these solutions of \eqref{eq:stationary_intro} are sometimes called least energy solutions, we will refer to them as \emph{least action solutions}. In case $\lambda$ is not given, one may define the ground states as the minimizers of $\mathcal{E}$ under the mass constraint $\mathcal{Q}(U)=\rho$, for some prescribed $\rho>0$ \cite[p. 555]{CazenaveLions1982}.
It is worth noticing that this second definition is fully consistent only in the subcritical case
\[
p<1+\frac{4}{N}
\]
since in the supercritical case $\mathcal{E}|_{\{\mathcal{Q}=\rho\}}$ is unbounded from below
\cite{Cazenave2003}, see also Appendix \ref{app:A}.
\begin{remark}\label{rem:kwong}
When working on the whole space $\R^N$, the two points of view above are in some sense equivalent.
Indeed, in such situation it is well known \cite{Kwong1989} that the problem
\begin{equation*}
-\Delta Z+Z=Z^p, \qquad Z\in H^1(\R^N),\ Z>0,
\end{equation*}
admits a solution $\Kwong$ which is unique (up to translations), radial and decreasing in $r$.
Therefore both the problem with fixed mass and the one with given chemical potential can be
uniquely solved in terms of a suitable scaling of $\Kwong$.
\end{remark}

When working on bounded domains, the two papers \cite{FibichMerle2001,Fukuizumi2012} mentioned above
deal with least action solutions. In this paper, we make a first attempt to study the case of
prescribed mass. Since we consider $p$ also in the critical and supercritical ranges, we have to
restrict the minimization process to constrained critical points of $\mathcal{E}$.

\begin{definition}\label{defi:les}
Let $\rho>0$. A positive solution of \eqref{eq:stationary_intro} with prescribed $L^2$-mass $\rho$ is a positive critical point of $\mathcal{E}|_{\{\mathcal{Q}=\rho\}}$, that is an element of the set
\[
\mathcal{P}_\rho=\left\{ U\in H^1_0(B_1):\ \mathcal{Q}(U)=\rho, \ U>0, \ \exists\lambda:\ -\Delta U+\lambda U=U^p \right\}.
\]
A positive \emph{least energy solution} is a minimizer of the problem
\[
e_\rho = \inf_{\mathcal{P}_\rho} \mathcal{E}.
\]
\end{definition}

\begin{remark}
When $p$ is subcritical, as we mentioned, the above procedure is equivalent to the minimization
of $\mathcal{E}|_{\{\mathcal{Q}=\rho\}}$ with no further constraint.
On the other hand, when $p$ is supercritical, the set $\mathcal{P}_\rho$ on which the
minimization is settled may be strongly irregular.
Contrarily to what happens for least action solutions, no natural Nehari manifold seems to be
associated to least energy solutions. Furthermore, since we work on a bounded
domain, the dependence of $\mathcal{P}_\rho$ on $\rho$ can not be understood in terms of dilations.
As a consequence, no regularized version of the minimization problem defined above seems available.
\end{remark}
\begin{remark}
Since $\mathcal{A}_\lambda$ and the corresponding Nehari manifold are even, it is immediate to see
that least action solutions do not change sign, so that they can be chosen to be positive.
On the other hand, since $U\in\mathcal{P}_\rho$ does not necessarily imply $|U|\in\mathcal{P}_\rho$,
in the previous definition we require the positivity of $U$. Nonetheless, this condition can be
removed in some cases, for instance when $p$ is subcritical, or when it is critical and $\rho$ is
small (see also Remark \ref{rem:segno_scontato}).
\end{remark}
Our main results deal with the existence and orbital stability of the least energy solutions of
\eqref{eq:stationary_intro}.
\begin{theorem}\label{thm:intro_existence}
Under the above notations, the following holds.
\begin{enumerate}
\item If $1<p<1+4/N$ then, for every $\rho>0$, the set $\mathcal{P}_\rho$ has a unique element, which achieves $e_\rho$;
\item if $p=1+4/N$, for $0<\rho< \|\Kwong\|_{L^2(\R^N)}^2$, the set $\mathcal{P}_\rho$ has a unique element, which achieves $e_\rho$; for $\rho\geq \|\Kwong\|_{L^2(\R^N)}^2$, we have $\mathcal{P}_\rho=\emptyset$ ;
\item if $1+4/N<p<2^*-1$, there exists $\rho^\ast>0$ such that $e_\rho$ is achieved if and only if $0<\rho\leq \rho^\ast$. Moreover, $\mathcal{P}_\rho=\emptyset$ for $\rho>\rho^\ast$, whereas
\[
\# \mathcal{P}_\rho\geq 2 \quad \text{ for } 0<\rho<\rho^\ast.
\]
In this latter case, $\mathcal{P}_\rho$ contains positive solutions of \eqref{eq:stationary_intro} which are not least energy solutions.
\end{enumerate}
\end{theorem}

\begin{remark}
As a consequence, we have that for $p$ and $\rho$ as in point 3. of the previous theorem, the problem
\[
\text{find } (U,\lambda)\in H^1_0(B_1)\times \R:\
\begin{cases}
-\Delta U+\lambda U=U^p\\
\int_{B_1} U^2\, dx=\rho
\end{cases}
\]
admits multiple positive radial solutions.
\end{remark}

Concerning the stability, following \cite{Fukuizumi2012}, we apply the abstract results in \cite{GrillakisShatahStrauss}, which require the local existence for the Cauchy problem associated to \eqref{eq:Schrodinger_intro}.
Since this is not known to hold for all the cases we consider, we take it as an assumption and refer to \cite[Remark 1]{Fukuizumi2012} for further details.

\begin{theorem}\label{thm:intro_stability}
Suppose that for each $\Phi_0\in H^1_0(B_1,\C)$ there exist $t_0>0$, only depending on $\|\Phi_0\|$,
and a unique solution $\Phi(t,x)$ of \eqref{eq:Schrodinger_intro} with initial datum $\Phi_0$, in
the interval $I=[0,t_0)$.

Let $U$ denote a least energy solution of \eqref{eq:stationary_intro} as in Theorem \ref{thm:intro_existence} and let $\Phi(t,x)=e^{i\lambda t}U(x)$.
\begin{enumerate}
\item If $1<p\leq 1+4/N$ then $\Phi$ is orbitally stable;
\item if $1+4/N<p<2^*-1$ then $\Phi$ is orbitally stable for a.e. $\rho \in (0,\rho^\ast]$.
\end{enumerate}
\end{theorem}
\noindent In point 2. of the previous theorem, we expect orbital stability for every $\rho\in (0,\rho^\ast)$,
and instability for $\rho=\rho^\ast$, see Remark \ref{rem:simulation} ahead.

As we mentioned, the authors in \cite{FibichMerle2001,Fukuizumi2012} consider least action solutions, that is minimizers associated to
\[
a_\lambda=\inf\{ \mathcal{A}_\lambda(U):\ U\in H^1_0(B_1),\ U\not\equiv 0, \
\mathcal{A}'_\lambda(U)=0 \}.
\]
In this situation, the existence and positivity of the least energy solution is not an issue. Indeed, it is well known that problem \eqref{eq:stationary_intro} admits a unique positive solution $R_\lambda$ if and only if $\lambda\in (-\lambda_1(B_1),+\infty)$, where $\lambda_1(B_1)$ is the first eigenvalue of the Dirichlet Laplacian. Such a solution achieves $a_\lambda$.
Concerning the stability, in the critical case \cite{FibichMerle2001} and in the subcritical one \cite{Fukuizumi2012} it is proved that $e^{i\lambda t}R_\lambda$ is orbitally stable whenever $\lambda \sim -\lambda_1(B_1)$ and $\lambda \sim +\infty$. Furthermore, stability for all $\lambda\in (-\lambda_1(B_1),+\infty)$ is proved in the second paper in dimension $N=1$ for $1<p\leq 5$, whereas in the first paper numerical evidence of it is provided in the critical case. In this context, our contribution is the following.

\begin{theorem}\label{thm:intro_stab_2}
Let us assume local existence as in Theorem \ref{thm:intro_stability} and let $R_\lambda$ be the
unique positive solution of \eqref{eq:stationary_intro}. If $1<p\leq 1+4/N$, then $e^{i\lambda t}R_
\lambda$ is orbitally stable for every $\lambda\in (-\lambda_1(B_1),+\infty)$.
\end{theorem}

\begin{remark}
In \cite{Fukuizumi2012} it is also shown that, in the supercritical case $p>1+4/N$, the standing wave associated to $R_\lambda$ is orbitally \emph{unstable} for $\lambda\sim +\infty$. In view of Theorem \ref{thm:intro_stability}, point 2., this marks a substantial qualitative difference between the two notions of ground state.
\end{remark}

We will prove the above results as a byproduct of the analysis of a different variational problem which we think is of independent interest. The main feature of such a problem lies on the fact that it involves an optimization with two constraints. Let $\Omega\subset \R^N$ be a general bounded domain. For every $\a>\l_1(\Omega)$ fixed, we consider the following maximization problem
\begin{equation}\label{eq:M_intro}
M_\a=\sup \left\{ \int_\Om |u|^{p+1}\,dx:\ u\in\spc, \int_\Om u^2\,dx=1,\ \int_\Om |\nabla u|^2\,dx=\a \right\},
\end{equation}
which is related to the validity of Gagliardo-Nirenberg type inequalities (Appendix \ref{app:A}).

\begin{theorem}\label{thm:intro_M}
Given $\alpha>\lambda_1(\Omega)$, $M_\alpha$ is achieved by a positive function $u\in H^1_0(\Omega)$, and there exist $\mu>0$, $\lambda>-\lambda_1(\Omega)$ such that
\begin{equation}\label{eq:problema_mu_lambda}
-\Delta u+\lambda u=\mu u^p,\quad \int_\Omega u^2\, dx=1,\quad \int_\Omega |\nabla u|^2\, dx=\alpha.
\end{equation}

Moreover, as $\alpha\to \lambda_1(\Omega)^+$,
\[
u\to \varphi_1,\quad \mu\to 0^+,\quad \lambda\to (-\lambda_1(\Omega))^+
\]
($\varphi_1$ denotes the first positive eigenfunction, normalized in $L^2$).

As $\alpha\to +\infty$,
\[\frac{\alpha}{\lambda}\to \frac{N(p-1)}{N+2-p(N-2)},\qquad \lambda\to +\infty,\]
and
\begin{enumerate}
 \item if $1<p<1+\frac{4}{N}$ then $\mu\to +\infty$;
 \item if $p=1+\frac{4}{N}$ then $\m \to \|\Kwong\|_{L^2(\R^N)}^{p-1}$;
 \item if $1+\frac{4}{N}<p<2^*-1$ then $\mu\to 0$;
\end{enumerate}
furthermore, $u$ is either a one-spike or two-spike solution, and a suitable scaling of $u$ approaches $\Kwong$, as defined in Remark \ref{rem:kwong}.
\end{theorem}

More detailed asymptotics are provided in Sections \ref{section:Ambrosetti-Prodi} and
\ref{sec:focusing}.
This problem is related to the previous one in the following way. Taking $u>0$ and $\mu>0$ as in
\eqref{eq:problema_mu_lambda},  the function $U=\mu^{1/(p-1)}u$ belongs to $\gsset_\rho$ for $\rho=
\mu^{2/(p-1)}$. Incidentally, if one considers the minimization problem
\[
m_\a=\inf \left\{ \int_\Om |u|^{p+1}\,dx:\ u\in\spc, \int_\Om u^2\,dx=1,\ \int_\Om |\nabla u|^2\,dx=
\a \right\},
\]
then one obtains a solution of \eqref{eq:problema_mu_lambda} with $\mu<0$ and $\lambda<-\lambda_1(\Omega)$. This allows to recover the well know theory of ground states for the defocusing Schr\"odinger equation $i\frac{\partial \Phi}{\partial t}+\Delta \Phi- |\Phi|^{p-1}\Phi=0$, see Appendix \ref{app:C}. Moreover, when $\alpha\sim \lambda_1(\Omega)$, there exist exactly two solutions $(u,\mu,\lambda)$ of \eqref{eq:problema_mu_lambda} which achieve $M_\alpha$ and $m_\alpha$ respectively. More precisely, in the context of Ambrosetti-Prodi theory \cite{AmbrosettiProdipaper,AmbrosettiProdiBook}, we prove that $(u,\mu,\lambda)=(\varphi_1,0,-\lambda_1(\Omega))$ is an ordinary singular point for a suitable map, which yields sharp asymptotic estimates as $\alpha\to \lambda_1(\Omega)^+$.

We stress the fact that the above mentioned results about the two-constraints problem hold for a
general bounded domain $\Omega$. Going back to the case $\Omega=B_1$,
positive solutions for equation \eqref{eq:stationary_intro} have been the object of an intensive
study by several authors, in particular regarding uniqueness issues; among others, we refer to
\cite{GidasNiNirenberg1979,Kwong1989,KwongLi1992,Zhang1992,KabeyaTanaka1999,Korman2002,Tang2003,FelmerMartinezTanaka2008}. In our framework, we can exploit the synergy with such uniqueness results in
order to fully characterize the positive solutions of \eqref{eq:problema_mu_lambda}. We do this in
the following statement, which collects the results of Proposition \ref{prop:focusing_radial_regular_curve} and of Appendix \ref{app:C} below.

\begin{theorem}\label{thm:main_final}
Let $\Omega=B_1$ and
\[
\mathcal{S}=\left\{(u,\mu,\lambda,\a)\in\spc\times\R^3 : u>0 \text{ and \eqref{eq:problema_mu_lambda} holds}\right\}.
\]
Then
\[\mathcal{S}=\mathcal{S}^+\cup \mathcal{S}^-\cup \{(\varphi_1,0,-\lambda_1(B_1),\lambda_1(B_1))\},\]
where both $\mathcal{S}^+$ and $\mathcal{S}^-$ are smooth curves parameterized by $\alpha\in (\lambda_1(B_1),+\infty)$, corresponding to $\mathcal{S}\cap\{\mu>0\}$ and $\mathcal{S}\cap\{\mu<0\}$ respectively. In addition, $(u,\mu,\lambda,\a)\in \mathcal{S}^+$ ($\mathcal{S}^-$) if and only if $u$ achieves $M_\a$ ($m_\alpha$).
\end{theorem}
\begin{remark}
As a consequence of the previous theorem, we have that the smooth set $\mathcal{S}^+$ defined
through the maximization problem $M_\alpha$ can be used as a surrogate of the Nehari manifold
in order to ``regularize'' the minimization procedure introduced in Definition \ref{defi:les}.
\end{remark}

This paper is structured as follows. In Section \ref{section:two_constraints} we address the preliminary study of the two-constraint problems associated to $M_\alpha$, $m_\alpha$. Afterwards, in Section \ref{section:Ambrosetti-Prodi} we focus on the case where $\alpha\sim \lambda_1(\Omega)$, seen as an Ambrosetti-Prodi type problem. Section \ref{sec:focusing} is devoted to the asymptotics as $\alpha\to +\infty$ for $M_\alpha$, which concludes the proof
of Theorem \ref{thm:intro_M}. In Section \ref{section:FocusingRadial}, we restrict our attention to the case $\Omega=B_1$, proving all the existence results (in particular Theorem \ref{thm:intro_existence}), qualitative properties, and more precise asymptotics for the map $\alpha\mapsto (u,\mu,\lambda)$ which parameterizes $\mathcal{S}^+$. In particular, we show that $\mu'(\alpha)>0$ whenever $p\leq 1+4/N$, whereas it changes sign in the supercritical case. Relying on such monotonicity properties, the stability issues are addressed in Section \ref{section:Stability}, which contains the proofs of Theorem \ref{thm:intro_stability} and Theorem \ref{thm:intro_stab_2}. Finally, in Appendices \ref{app:A}, \ref{app:B} we collect some known results for the reader's convenience, whereas Appendix \ref{app:C} is devoted to the study of $\mathcal{S}^-$, which concludes the proof of Theorem \ref{thm:main_final}.
%
%
%
%
\section{A variational problem with two constraints}\label{section:two_constraints}
%
%
Let $\Omega \subset \R^N$ be a bounded domain, $N\geq 1$. For every $\a\geq\l_1(\Omega)$ fixed, we consider the following variational problems
\[
m_\a=\inf_{u\in\mathcal{U}_\a} \int_\Om |u|^{p+1}\,dx, \qquad M_\a=\sup_{u\in\mathcal{U}_\a} \int_\Om |u|^{p+1}\,dx,
\]
where
\[
\mathcal{U}_\a=\left\{ u\in\spc: \int_\Om u^2\,dx=1,\ \int_\Om |\nabla u|^2\,dx\leq\a \right\}.
\]
As we will see, these definitions of $M_\alpha$ and $m_\alpha$ are equivalent to the ones given in the Introduction.  To start with, we state the following straightforward properties.

\begin{lemma}\label{lemma:U_not_empty_and_compact}
For every fixed $\a\geq\l_1(\Omega)$ it holds
\begin{itemize}
\item[(i)] $\mathcal{U}_\a\neq\emptyset$;
\item[(ii)] $\mathcal{U}_\a$ is weakly compact in $\spc$;
\item[(iii)] the functional $u\mapsto\int_\Om |u|^{p+1}\,dx$ is weakly continuous and bounded in $\mathcal{U}_\a$;
\item[(iv)] $\|u\|_{L^{p+1}(\Omega)}\geq|\Omega|^{-\frac{p-1}{2(p+1)}}$ for every $u\in\mathcal{U}_\a$. 
\end{itemize}
\end{lemma}

\begin{lemma}\label{lemma:tilde_U_manifold}
For every fixed $\a>\l_1(\Omega)$ the set
\[
\tilde{\mathcal{U}}_\a=\left\{u\in\spc: \int_\Om u^2\,dx=1, \ \int_\Om|\nabla u|^2\,dx=\a, \
\int_\Om u\vphi_1\,dx\neq0 \right\}
\]
is a submanifold of $H^1_0(\Omega)$ of codimension 2.
\end{lemma}
\begin{proof}
Setting $F(u)=(\int_\Om u^2\,dx-1, \ \int_\Om|\nabla u|^2\,dx)$, it suffices to prove that, for every $u\in\tilde{\mathcal{U}}_\a$, the range of $F'(u)$ is $\R^2$. We have
\[
\frac{F'(u)[u]}{2}=(1,\a), \qquad \frac{F'(u)[\vphi_1]}{2}=\int_\Om u\vphi_1\,dx\cdot(1,\l_1(\Omega)),
\]
which are linearly independent as $\alpha>\lambda_1(\Omega)$.
\end{proof}

\begin{lemma}\label{lemma:m_M_attained}
For every fixed $\a>\l_1(\Omega)$ there exists $u\in \tilde{\mathcal{U}}_\a$, with $u\geq0$, such that $m_\alpha=\int_\Om u^{p+1}\,dx$.
Moreover there exist $\l,\m\in\R$, with $\mu\neq0$, such that
\begin{equation}\label{eq:mu_lambda}
-\Delta u+\lambda u=\m u^p \quad\text{in }\Om.
\end{equation}
A similar result holds for $M_\a$.
\end{lemma}
\begin{proof}
Let us proof the result for $m_\a$. First, the infimum is attained by a function $u\in\mathcal{U}_\a$ by Lemma \ref{lemma:U_not_empty_and_compact}; by possibly taking $|u|$, we can suppose that $u\geq0$. Let us show that $u\in\tilde{\mathcal{U}}_\a$. Notice that, being $u\geq0$ and $u\not\equiv0$, it holds $\int_\Om u\vphi_1\,dx\neq0$. Assume by contradiction that $\int_\Om|\nabla u|^2\,dx<\a$, then we have
\[
\int_\Om u^{p+1}\,dx=\inf\left\{ \int_\Om |v|^{p+1}\,dx:\ v\in\spc, \ \int_\Om v^2\,dx=1,\ \int_\Omega |\nabla u|^2\, dx<\alpha\right\},
\]
and there exists a Lagrange multiplier $\m\in \R$ so that
\[
\int_\Om u^pz\,dx=\m\int_\Om uz\,dx, \quad\text{for all } z\in\spc.
\]
Hence $\m\equiv u^{p-1}\in \spc$, which contradicts the fact that $\int_\Om u^2\,dx=1$. Therefore $u\in\tilde{\mathcal{U}}_\a$ so that, by Lemma \ref{lemma:tilde_U_manifold}, the Lagrange multiplier theorem applies, thus providing the existence of $k_1,k_2\in\R$ such that
\[
\int_\Omega u^pz\,dx=k_1\int_\Omega\nabla u\cdot \nabla z\,dx+k_2\int_\Omega uz \,dx \quad\text{for all } z\in H^1_0(\Omega).
\]
By the previous argument we see that $k_1\neq0$, hence setting $\m=1/k_1$ and $\l=k_2/k_1$, the proposition is proved.
\end{proof}

\begin{proposition}\label{prop:sign_of_mu_lambda}
Given $\alpha>\lambda_1(\Omega)$, the Lagrange multipliers $\mu,\lambda$ associated to $m_\a$ as in Lemma \ref{lemma:m_M_attained} satisfy $\m<0$, $\lambda<-\lambda_1(\Omega)$. Similarly, in the case of $M_\a$, it holds $\m>0$, $\lambda>-\lambda_1(\Omega)$.
\end{proposition}
\begin{proof}
Let $(u,\l,\m)$ be any triplet associated to $m_\a$ as in Lemma \ref{lemma:m_M_attained}. We will prove that
$\mu<0$. Set
\[
w(t)=tu+s(t)\vphi_1,
\]
where $t\in\R$ is close to $1$, $s(1)=0$, and $s(t)$ is such that
\begin{equation}\label{eq:sign_of_mu0}
1=\int_\Om w(t)^2\,dx=t^2+2t s(t)\int_\Om u\vphi_1\,dx+s(t)^2.
\end{equation}
Since
\[
\left.\partial_s\left(t^2+2ts\int_\Om u\vphi_1\,dx+s^2\right)\right|_{(t,s)=(1,0)}=2\int_\Om u\vphi_1\,dx\neq0,
\]
the Implicit Function Theorem applies, providing that the map $t\mapsto w(t)$ is of class $C^1$ in a neighborhood of $t=1$. Differentiating \eqref{eq:sign_of_mu0} with respect to $t$ at $t=1$, we obtain
\begin{equation*}
0=\int_\Om w'(1)w(1)\,dx=\int_\Om w'(1)u\,dx = 1 + s'(1)\int_\Omega u \vphi_1\, dx,
\end{equation*}
which implies $s'(1)=-1/\int_\Omega u\vphi_1\,dx$ and
$w'(1)=u-\vphi_1/\int_\Omega u\vphi_1\,dx$. Thus
\begin{equation}\label{eq:sign_of_mu3}
\begin{split}
\left.\frac{1}{2}\frac{d}{dt} \int_\Omega |\nabla w(t)|^2\, dx\right|_{t=1}&= \int_\Om\nabla u\cdot\nabla w'(1)\,dx \\
& = \int_\Omega |\nabla u|^2\, dx-\frac{\int_\Omega \nabla u\cdot \nabla \vphi_1\, dx}{\int_\Omega u\vphi_1\, dx}=\alpha-\lambda_1(\Omega)>0.
\end{split}
\end{equation}
In particular, this implies the existence of $\eps>0$ such that $w(t)\in\mathcal{U}_\a$ for $t\in (1-\eps, 1]$. Therefore, by the definition of $m_\a$, $\|w(1)\|_{p+1}\leq\|w(t)\|_{p+1}$ for every $t\in (1-\eps,1]$,
and
\begin{equation}\label{eq:sign_of_mu4}
\frac{d}{dt}\int_\Om|w(t)|^{p+1}\,dx|_{t=1}\leq 0
\end{equation}
On the other hand, using  \eqref{eq:mu_lambda} and the fact that $\int_\Omega uw'(1)\,dx=0$, we have
\[
\begin{split}
\left.\frac{\m}{p+1} \frac{d}{dt}\int_\Om|w(t)|^{p+1}\,dx\right|_{t=1}
&=\m\int_\Om u^p w'(1)\,dx=\int_\Om(-\Delta u + \lambda u)w'(1)\,dx\\
= &\int_\Om\nabla u\cdot\nabla w'(1)\,dx
=\left.\frac{1}{2}\frac{d}{dt}\int_\Om|\nabla w(t)|^2\,dx\right|_{t=1}>0.
\end{split}
\]
by \eqref{eq:sign_of_mu3}. By comparing with \eqref{eq:sign_of_mu4} we obtain that $\m<0$.

The case of $M_\a$ can be handled in the same way, obtaining that in such situation $\mu>0$.
Finally, by multiplying equation \eqref{eq:mu_lambda} by $\vphi_1$, we obtain
\[
(\lambda_1(\Omega)+\lambda)\int_\Omega u \vphi_1\, dx= \mu \int_\Omega u^p \vphi_1 \,dx.
\]
As $u,\vphi_1\geq 0$, we deduce that $\lambda_1(\Omega)+\lambda$ has the same sign than $\mu$.
\end{proof}



We conclude this section with the following boundedness result, which we will need later on.

\begin{lemma}\label{lemma:case_alpha_n_bounded}
Take a sequence $\{(u_n,\mu_n,\lambda_n)\}_n$ such that
$$
\int_\Omega u_n^2\, dx=1,\qquad \int_\Omega |\nabla u_n|^2\, dx=:\alpha_n \text{ is bounded,}
$$
and
\begin{equation}\label{eq:auxiliary_n}
-\Delta u_n+\lambda_n u_n=\mu_n u_n^p.
\end{equation}
Then the sequences $\{\lambda_n\}_n$ and $\{\mu_n\}_n$ are bounded.
\end{lemma}
\begin{proof}
By multiplying \eqref{eq:auxiliary_n} by $u_n$ we see that
$$
\alpha_n+\lambda_n=\mu_n\int_\Omega u_n^{p+1}\, dx,
$$
thus if one the sequences $\{\lambda_n\}_n,\ \{\mu_n\}_n$ is bounded, the other is also bounded. Recall that, by assumption, $u_n$ is bounded in $H^1_0(\Omega)$, hence it converges in the $L^{p+1}$-norm to some $u\in H^1_0(\Omega)$, up to a subsequence. Moreover $u\not\equiv 0$, as $\int_\Omega u^2\, dx=1$.

To fix ideas, suppose without loss of generality that $\mu_n\to +\infty$ and that $\lambda_n\to +\infty$. From the previous identity we also have that
$$
\frac{\lambda_n}{\mu_n}=\int_\Omega u_n^{p+1}\, dx-\frac{\alpha_n}{\mu_n} \to \int_\Omega u^{p+1}\, dx=: \gamma\neq0,
$$
up to a subsequence. Now take any $\vphi\in H^1_0(\Omega)$ and use it as test function in \eqref{eq:auxiliary_n}. We obtain
\[
\begin{split}
\int_\Omega \nabla u_n\cdot \nabla \vphi\, dx &=\mu_n\int_\Omega u_n^p\vphi\, dx-\lambda_n\int_\Omega u_n\vphi\, dx\\
&=\mu_n \left(\int_\Omega u_n^p \vphi\,dx-\frac{\lambda_n}{\mu_n}\int_\Omega u_n \vphi\, dx\right).
\end{split}
\]
As $\mu_n\to +\infty$, we must have
$$
\int_\Omega u_n^p \vphi\, dx-\frac{\lambda_n}{\mu_n}\int_\Omega u_n \vphi\, dx\to 0.
$$
On the other hand,
$$
\int_\Omega u_n^p\vphi\,dx -\frac{\lambda_n}{\mu_n}\int_\Omega u_n \vphi\, dx\to
 \int_\Omega u^p\vphi\, dx-\gamma\int_\Omega u \vphi \,dx.
$$
Thus we have $u^p\equiv \gamma u$, which is a contradiction.
\end{proof}

\section{Asymptotics as $\alpha\to \lambda_1(\Omega)^+$}\label{section:Ambrosetti-Prodi}

In this section we will completely describe the solutions of the problem
\begin{equation}\label{eq:main_problem}
-\Delta u+\lambda u=\mu u^p, \ u\in H^1_0(\Omega),\ u>0,\qquad \int_\Omega u^2\, dx=1,
\end{equation}
for $\alpha:=\int_\Omega |\nabla u|^2\, dx$ in a (right) neighborhood of $\lambda_1(\Omega)$. For that we will follow the theory presented in \cite[Section 3.2]{AmbrosettiProdiBook}, which we now briefly recall.

\begin{definition}
Let $X,Y$ be Banach spaces, $U\subseteq X$ an open set and $\Phi\in C^2(U,Y)$. A point $x\in U$ is said to be \emph{ordinary singular} for $\Phi$ if
\begin{itemize}
\item[(a)] $\Ker (\Phi'(x))$ is one dimensional, spanned by a certain $\phi\in X$;
\item[(b)] $R(\Phi'(x))$ is closed and has codimension one;
\item[(c)] $\Phi''(x)[\phi,\phi]\notin R(\Phi'(x))$;
\end{itemize}
where $\Ker(\Phi'(x))$ and $R(\Phi'(x))$ denote respectively the kernel and the range of the map $\Phi'(x):X\to Y$.
\end{definition}

We will need the following result.

\begin{theorem}[{\cite[Section 3.2, Lemma 2.5]{AmbrosettiProdiBook}}]\label{thm:AmbrosettiProdi}
Under the previous notations, let $x^\ast\in U$ be an ordinary singular point for $\Phi$. Take $y^\ast=\Phi(x^\ast)$, $\phi\in X$ such that $\Ker(\Phi'(x^\ast))=\R \phi$, $\Psi\in Y^\ast$ such that $R(\Phi'(x^\ast))=\Ker (\Psi)$ and consider $z\in Z$ such that $\Psi(z)=1$, where $Y=Z\oplus \Ker(\Phi'(x^\ast))$. Suppose that
$$
\Psi(\Phi''(x^\ast)[\phi,\phi])>0.
$$
Then there exist $\eps^\ast,\delta>0$ such that the equation
\[
\Phi(x)=y^\ast+\eps z, \qquad x\in B_\delta(x^\ast)
\]
has exactly two solutions for each $0<\eps<\eps^\ast$, and no solutions for all $-\eps^\ast<\eps<0$. Moreover, there exists $\sigma>0$ such that the solutions can be parameterized with a parameter $t\in (-\sigma, \sigma)$, $t\mapsto x(t)$ is a $C^1$ map and
\begin{equation}\label{eq:ambrosetti_prodi_expansion}
x(t)=x^*+t\phi+o(\sqrt{\eps}) \qquad \text{ with } \quad t=\pm\sqrt{\frac{2\eps}{ \Psi(\Phi''(x^*)[\phi,\phi])}}.
\end{equation}
\end{theorem}

Let us now set the framework which will allow us to apply the previous results to our situation. Given $k>N$, consider $X=\{w\in W^{2,k}(\Omega):\ w=0 \text{ on } \partial \Omega\}$, $Y=L^k(\Omega)$ and $U=\{w\in X:\ w>0 \text{ in } \Omega \text{ and } \partial_\nu w<0 \text{ on } \partial \Omega\}$. Take $\Phi:X\times \R^2\to L^k(\Omega)\times \R^2$ defined by
\begin{equation}\label{eq:Phi_definition}
\Phi(u,\mu,\lambda)=\left(\Delta u - \lambda u+ \mu u^p ,\ \int_\Omega u^2\, dx-1,\ \int_\Omega |\nabla u|^2\, dx\right).
\end{equation}
\begin{remark}\label{rem:fisitu}
$\Phi\in C^2(U,Y)$. This is immediate when $p\geq2$, and it also holds true for $1<p<2$. We postpone to Appendix \ref{app:B} the proof of this fact.
\end{remark}
We start with the following result.
\begin{lemma}\label{lemma:alpha_to_lambda_1}
Let $\alpha_n\to \lambda_1(\Omega)^+$ and suppose there exists $(u_n,\mu_n,\lambda_n)$ such that $\Phi(u_n,\mu_n,\lambda_n)=(0,0,\alpha_n)$ with $u_n\geq 0$. Then $u_n\to \vphi_1$ in $H^1_0(\Omega)$, $\mu_n\to 0$ and $\lambda_n\to -\lambda_1(\Omega)$. In particular,
$$
\Phi(u,\mu,\lambda)=(0,0,\lambda_1(\Omega)),\ u\geq 0\quad \text{ if and only if } (u,\mu, \lambda)=(\vphi_1,0,-\lambda_1(\Omega)).
$$
\end{lemma}
\begin{proof}
As $u_n$ is bounded in $H^1_0(\Omega)$, then up to a subsequence we have that $u_n \rightharpoonup u$ weakly in $H^1_0(\Omega)$. Moreover $\int_\Omega u^2\, dx=1$, $u\geq 0$ and, by the Poincar\'e inequality, $\lambda_1(\Omega)\leq \int_\Omega |\nabla u|^2\leq \liminf \int_\Omega |\nabla u_n|^2\, dx=\lambda_1(\Omega)$, whence $u=\vphi_1$ and the whole sequence $u_n$ converges strongly to $\vphi_1$ in $H^1_0(\Omega)$. By Lemma \ref{lemma:case_alpha_n_bounded} we have that $\mu_n$ and $\lambda_n$ are bounded. Denote by $\mu_\infty$ and $\lambda_\infty$ a limit of one of its subsequences. Then
$$
-\Delta \vphi_1+\lambda_\infty \vphi_1=\mu_\infty \vphi_1^p,
$$
which shows that $\mu_\infty=0$, $\lambda_\infty=-\lambda_1(\Omega)$.
\end{proof}

\begin{lemma} The point $(\vphi_1,0,-\lambda_1(\Omega))\in U$ is ordinary singular for $\Phi$. More precisely, for $L:= \Phi'(\vphi_1,0,-\lambda_1(\Omega)):X\times \R^2\to L^k(\Omega)\times \R^2$, we have
\begin{itemize}
\item[(i)] $\Ker(L)= \mathrm{span}\left\{ \left(\psi,\ 1,\ \int_\Omega \vphi_1^{p+1}\,dx\right)\right\}=:\mathrm{span}\{\phi\}$, where $\psi\in X$ is the unique solution of
\begin{equation}\label{eq:psi}
-\Delta \psi-\lambda_1(\Omega) \psi=\vphi_1^p-\vphi_1 \int_\Omega \vphi_1^{p+1}\,dx, \ \text{ such that }
\int_\Omega \psi \vphi_1\, dx=0;
\end{equation}
\item[(ii)] $R(L)=\Ker(\Psi)$, with $\Psi:L^k(\Omega)\times \R^2\to \R$ defined by $\Psi(\xi,h,k)=k-\lambda_1(\Omega)h$;
\item[(iii)] $\Psi(\Phi''(\vphi_1,0,-\lambda_1(\Omega))[\phi,\phi])>0.$
\end{itemize}
\end{lemma}

\begin{proof}
(i) We recall that $-\Delta -\l_1(\Omega) Id$ is a Fredholm operator of index $0$, with
\[
\begin{split}
&\Ker(-\Delta-\lambda_1(\Omega) Id)=\mathrm{span} \{\varphi_1\},\\
&R(-\Delta-\l_1(\Omega) Id) =\left\{v\in L^k(\Omega):\ \int_\Omega v\vphi_1\,dx=0\right\}.
\end{split}
\]
Therefore, by the Fredholm alternative, there exists a unique $\psi\in X$ solution of \eqref{eq:psi}.
Let us check that $\Ker(L)= \mathrm{span}\left\{ \left(\psi,\ 1,\ \int_\Omega \vphi_1^{p+1}\,dx\right)\right\}$.
We have
\begin{equation*}
L(v,m,l)=\left(\Delta v+\lambda_1(\Omega) v - l \vphi_1+m \vphi_1^p,\ 2\int_\Omega \vphi_1 v\, dx,\ 2 \int_\Omega \nabla \vphi_1\cdot \nabla v\, dx\right),
\end{equation*}
thus $(v,m,l)\in Ker(L)$ if and only if $l=m\int_\Omega \vphi_1^{p+1}$, $\int_\Omega \vphi_1 v\, dx=\int_\Omega \nabla \vphi_1\cdot \nabla v\, dx=0$, and
\[
-\Delta v-\lambda_1(\Omega) v=m\left(\vphi_1^p-\vphi_1\int_\Omega \vphi_1^{p+1}\,dx\right) \ \text{ for some } m\in \R.
\]
By the uniqueness of $\psi$ in \eqref{eq:psi}, we obtain $v=m \psi$.

(ii) Let us prove that $R(L)=\{(\xi,h,\l_1(\Omega) h):\ \xi\in L^k(\Omega),\ h\in\R\}$. Recalling the expression for $L$ found in (i), it is clear that $L(v,m,l)=(\xi,h,k)$ implies $k=\l_1(\Omega) h$. As for the other inclusion, given any $\xi\in L^k(\Omega)$, let $w\in X$ be the solution of
\[
-\Delta w-\lambda_1(\Omega) w=\vphi_1 \int_\Omega \xi \vphi_1\, dx -\xi, \ \text{ with }
\int_\Omega w \vphi_1\, dx=0,
\]
which exists and is unique again by the Fredholm alternative. Then $L(h\vphi_1/2+w,0,\int_\Omega \xi\vphi_1\,dx)=(\xi,h,\lambda_1(\Omega) h)$.

(iii) We have that
\[
\Phi''(\varphi_1,0,-\lambda_1(\Omega))[\phi,\phi] =2(p\vphi_1^{p-1}\psi-\psi\int_\Omega \vphi_1^{p+1}\,dx,\ \int_\Omega \psi^2\,dx,\ \int_\Omega|\nabla\psi|^2\,dx),
\]
with $\phi$ and $\psi$ defined at point (i). Hence
\begin{equation}\label{eq:ambrosetti_prodi_expansion_application}
\Psi(\Phi''(\varphi_1,0,\lambda_1(\Omega))[\phi,\phi])=\int_\Omega 2(|\nabla\psi|^2-\lambda_1(\Omega)\psi^2)\,dx>0,
\end{equation}
since $\psi$ satisfies \eqref{eq:psi}.
\end{proof}

\begin{proposition}\label{prop:neighborhood of lambda_1}
There exists $\eps^*$ such that the equation
\[
\Phi(u,\mu,\lambda)=(0,0,\lambda_1(\Omega)+\eps),\qquad (u,\mu,\lambda)\in U\times \R^2
\]
has exactly two positive solutions for each $0<\eps<\eps^*$ (one with $\mu>0$ and one with $\mu<0$). Moreover, such solutions satisfy the asymptotic expansion
\[
(u,\mu,\lambda)=(\varphi_1,0,-\lambda_1(\Omega))\pm \sqrt{\frac{\eps}{\int_\Omega \vphi_1^p\psi\,dx}} \left(\psi,1,\int_\Omega \varphi_1^{p+1}\,dx\right)+o(\sqrt{\eps}),
\]
where $\psi$ is defined in \eqref{eq:psi}. In addition, the $L^{p+1}$--norm of one of the solutions is equal to $m_{\lambda_1(\Omega)+\eps}$ and the other is equal to $M_{\lambda_1(\Omega)+\eps}$. 
\end{proposition}
\begin{proof}
We apply Theorem \ref{thm:AmbrosettiProdi} with $\Phi$ defined in \eqref{eq:Phi_definition}, $x^*=(\vphi_1,0,-\lambda_1(\Omega))$ and $z=(0,0,1)$. By the previous lemma, $x^*$ is ordinary singular for $\Phi$ and moreover, using the notations therein, $\Psi(\Phi''(x^\ast)[\phi,\phi])>0 $ and $\Psi(z)=1$. Therefore the assumptions of Theorem \ref{thm:AmbrosettiProdi} are satisfied and there exists $\eps^\ast,\delta>0$ such that the problem
\[
\Phi(u,\mu,\lambda)=(0,0,\lambda_1(\Omega)+\eps),\qquad (u,\mu,\lambda)\in B_\delta(\varphi_1,0,-\lambda_1(\Omega))
\]
has exactly two solutions for each $0<\eps<\eps^\ast$, which can be parameterized using a map $t \mapsto (u(t), \mu(t), \l(t))$ of class $C^1$ in $U\times\R^2$. The asymptotic expansion is obtained by combining \eqref{eq:ambrosetti_prodi_expansion} with the fact that (cf. \eqref{eq:ambrosetti_prodi_expansion_application})
\[
\Psi(\Phi''(\varphi_1,0,\lambda_1(\Omega))[\phi,\phi])=2\int_\Omega \varphi_1^p\psi\,dx.
\]
Finally, by possibly choosing a smaller $\eps^\ast$, $(u(t), \mu(t), \l(t))$ are the unique positive solutions in $U\times \R^2$ for $0<\eps<\eps^\ast$, as a consequence of Lemma \ref{lemma:alpha_to_lambda_1}, and the statement concerning $\int_\Omega u(t)^{p+1} \,dx$ follows from Lemma \ref{lemma:m_M_attained} and Proposition \ref{prop:sign_of_mu_lambda}.
\end{proof}

\begin{remark}
From Proposition \ref{prop:neighborhood of lambda_1} we deduce an alternative proof of
\cite[Theorem 17 (ii)]{Fukuizumi2012}, namely we can show that
\begin{equation*}
(\mu^2)'(\lambda_1(\Omega)^+)>0.
\end{equation*}
This result is relevant when facing stability issues, see Corollary \ref{coro:mu'>0=>stab} ahead.
\end{remark}


\section{Asymptotics as $\a\to+\infty$}\label{sec:focusing}

In this section we consider the case when $\a$ is large, in order to conclude the proof of Theorem \ref{thm:intro_M}. Since in such case the problems $M_\a$ and $m_\a$ exhibit different asymptotics, here we only address the study of $M_\a$, and we postpone to Appendix \ref{app:C} the complete description of the minimizers corresponding to $m_\a$.

Define, for any $\mu,\lambda\in \R$, the action functional associated to \eqref{eq:mu_lambda}, namely $J_{\mu,\lambda}:H^1_0(\Omega)\to \R$,
\begin{equation}\label{eq:Jmulambda}
J_{\mu,\lambda}(u)=\frac{1}{2}\int_\Omega\left(|\nabla u|^2+\lambda u^2\right)\, dx-\frac{\mu}{p+1}\int_\Omega |u|^{p+1}\, dx.
\end{equation}

\begin{lemma}\label{lemma:M_alpha_minimization}
For every $\mu,\l\in\R$ we have that
\[
u\in \tilde{\mathcal{U}}_\a \text{ and } \int_\Omega |u|^{p+1}\,dx=M_\a \quad \text{ implies }\quad
J_{\mu,\l}(u)=\inf_{\tilde{\mathcal{U}}_\a} J_{\mu,\l}.
\]
\end{lemma}
\begin{proof}
By definition of $M_\a$ it holds
\begin{multline*}
\frac{\mu}{p+1}M_\a=\sup_{w\in \tilde{\mathcal{U}}_\a}
\left\{ \frac{\mu}{p+1}\int_\Omega |w|^{p+1}\,dx +
\right.\\ \left.
+\frac{1}{2}\left(\a-\int_\Omega|\nabla w|^2\,dx\right) +
\frac{\l}{2}\left(1-\int_\Omega w^2\,dx\right) \right\}
\end{multline*}
and hence
\[
J_{\mu,\lambda}(u)=\frac{\a+\l}{2}-\frac{\mu}{p+1}M_\a=\inf_{w\in \tilde{\mathcal{U}}_\a} J_{\mu,\l}(w). \qedhere
\]
\end{proof}

\begin{lemma}\label{lemma:morse_index_estimate}
Fix $\a>\l_1(\Omega)$ and let $(u,\mu,\l)\in \tilde{\mathcal{U}}_\a\times\R^+\times(-\lambda_1(\Omega),+\infty)$ be any triplet associated to $M_\a$ as in Lemma \ref{lemma:m_M_attained}. Then the Morse index of $J_{\mu,\l}''(u)$ is either 1 or 2.
\end{lemma}
\begin{proof}
If $(u,\mu,\l)$ is a triplet associated to $M_\a$, then $\mu>0$ by Proposition \ref{prop:sign_of_mu_lambda}.
Equation \eqref{eq:mu_lambda} implies
\[
J_{\mu,\l}''(u)[u,u]=-(p-1)\m \int_\Omega u^{p+1}\,dx<0,
\]
so that the Morse index is at least 1.
Next we claim that for such $(u,\mu,\l)$ it holds
\[
J_{\mu,\l}''(u)[\phi,\phi]\geq 0, \quad \text{ for every } \phi\in H^1_0(\Omega) \text{ with }
\int_\Omega \nabla u\cdot \nabla\phi \,dx=\int_\Omega u\phi \, dx=0,
\]
which implies that the Morse index is at most 2. Indeed, any such $\phi$ belongs to the tangent space of $\tilde{\mathcal{U}}_\a$ at $u$, hence there exists a $C^\infty$ curve $\gamma(t)$ satisfying, for some $\eps>0$,
\[
\gamma:(-\eps,\eps)\to \tilde{\mathcal{U}}_\a,\quad \gamma(0)=u,\quad \gamma'(0)=\phi.
\]
Lemma \ref{lemma:M_alpha_minimization} implies that $J_{\mu,\l}(\gamma(t))-J_{\mu,\l}(\gamma(0))\geq0$.
Hence
\[
0\leq J_{\mu,\l}(\gamma(t))-J_{\mu,\l}(u)=J_{\mu,\l}'(u)[\phi]t+J_{\mu,\l}''(u)[\phi,\phi]\frac{t^2}{2} +
J_{\mu,\l}'(u)[\gamma''(0)]\frac{t^2}{2}+o(t^2)
\]
Finally, equation \eqref{eq:mu_lambda} implies that $J_{\mu,\l}'(u)\equiv0$, which concludes the proof.
\end{proof}

Next, we use some results from \cite{EspositoPetralla2011} in order to show that a suitable rescaling of the solutions converges to the function $\Kwong$, defined in Remark \ref{rem:kwong}. This will allow us to study the asymptotic behavior of $\mu$ as $\alpha\to +\infty$.
\begin{lemma}\label{rem:blowup}
Take $\alpha_n\to+\infty$ and let $u_n\in H^1_0(\Omega)$, $u_n>0$ satisfy
\[
-\Delta u_n+\l_n u_n =\mu_n u_n^p \quad \text{in }\Omega, \qquad \int_\Omega |\nabla u_n|^2\,dx=\alpha_n, \qquad \int_\Omega u_n^2\,dx=1,
\]
for some $\mu_n>0$ and $\l_n>-\l_1(\Omega)$. Suppose moreover that the Morse index of $J''_{\mu,\lambda}(u_n)$ is uniformly bounded in $n$. Let $x_n$ be a local maximum for $u_n$, and define
\begin{equation}\label{eq:rescailing}
v_n(x)=\left(\frac{\mu_n}{\lambda_n}\right)^{1/(p-1)}u_n\left(\frac{x}{\sqrt{\lambda_n}}+x_n\right),\text{ for } x\in \Omega=\sqrt{\lambda_n}(\Omega-x_n).
\end{equation}
Then
$$
v_n\to \Kwong \quad \text{ in } H^1(\R^N) \qquad \text{ as } n\to +\infty.
$$
\end{lemma}

\begin{proof}
By possibly working with a subsequence, we can suppose without loss of generality that the Morse index of $J_{\mu,\lambda}''(u_n)$ is equal to $k$. The function $v_n$ solves $-\Delta v_n+v_n=v_n^p$ in $\Omega_n$, and Theorem 3.1 of \cite{EspositoPetralla2011} yields that $v_n\to \Kwong$ in $C^1_\textrm{loc}(\R^N)$. Next, applying \cite[Theorem 3.2]{EspositoPetralla2011} to $U_n:=\mu_n^{1/(p-1)}u_n$, we have the existence of $k$ local maxima $P_n^i$, $i=1,\ldots, k$ so that
\[
\sqrt{\lambda_n}|P_n^i-P_n^j|\to +\infty \qquad \text{ whenever }i\neq j,
\]
and
\[ U_n(x)\leq C \lambda_n^{1/(p-1)}\sum_{i=1}^ke^{-\gamma\sqrt{\lambda_n}|x-P_n^i|},\qquad \forall x\in \Omega.
\]
Thus
\[
v_n(x)\leq C\sum_{i=1}^k e^{-|\sqrt{\lambda_n}(P_n^j-P_n^i)+x|},\qquad \forall x\in \Omega_n
\]
and $v_n$ decays exponentially to 0 as $|x|\to +\infty$, uniformly in $n$. It is now straightforward to conclude.
\end{proof}

\begin{lemma}\label{lemma:mu_tends_to_zero}
In the same assumptions of the previous lemma, we have that
\[\frac{\alpha_n}{\lambda_n}\to \frac{N(p-1)}{N+2-p(N-2)},\qquad \lambda_n\to+\infty,\]
and
\begin{enumerate}
 \item if $1<p<1+\frac{4}{N}$ then $\mu_n\to +\infty$;
 \item if $p=1+\frac{4}{N}$ then $\m_n \to \|\Kwong\|_{L^2(\R^N)}^{p-1}$;
 \item if $1+\frac{4}{N}<p<2^*-1$ then $\mu_n\to 0$.
\end{enumerate}
The result holds in particular in case $(u_n,\mu_n,\l_n)$ is a triplet which achieves $M_{\a_n}$, with $\alpha_n\to+\infty$.
\end{lemma}

\begin{proof}
Let $x_n$ be so that $u_n(x_n)=\|u_n\|_{L^\infty(\Omega)}=:L_n$ and define $v_n$ as in \eqref{eq:rescailing}. Then Lemma \ref{rem:blowup} implies that
\begin{equation}\label{eq:integral_of_v_n^2}
\mu_n^{2/(p-1)}\lambda_n^{N/2-2/(p-1)}=\int_{\Omega_n} v_n^2\, dx\to \int_{\R^N}\Kwong^2\, dx
\end{equation}
and
\[
\frac{\alpha_n}{\lambda_n}\mu_n^{2/(p-1)}\lambda_n^{N/2-2/(p-1)}=\int_{\Omega_n} |\nabla v_n|^2 \, dx\to
\int_{\R^N}|\nabla \Kwong|^2\,dx,
\]
Thus
\begin{equation}\label{eq:limit_of_alpha_lambda}
\frac{\alpha_n}{\lambda_n}\to \frac{\|\nabla \Kwong\|_{L^2(\R^N)}^2}{\|\Kwong\|_{L^2(\R^N)}^2}=\frac{N(p-1)}{N+2-p(N-2)}.
\end{equation}
where the last relation follows by combining the Pohozaev identity with the equality
$\|\nabla \Kwong\|_{L^2(\R^N)}^2+\|\Kwong\|_{L^2(\R^N)}^2=\|\Kwong\|_{L^{p+1}(\R^N)}^{p+1}$.

From \eqref{eq:limit_of_alpha_lambda} we have that, as $\alpha\to +\infty$, $\lambda\to +\infty$. Combining this information with the fact that the exponent $N/2-2/(p-1)$ is negative, zero, or positive respectively in the subcritical, critical, and supercritical case, the properties for $\mu_n$ follow from \eqref{eq:integral_of_v_n^2}.
\end{proof}


\begin{proof}[End of the proof of Theorem \ref{thm:intro_M}]
The fact that $M_\alpha$ is achieved by a triplet $(u,\mu,\lambda)$, with $\mu>0$ and $\lambda>-\lambda_1(\Omega)$ is a consequence of Lemma \ref{lemma:m_M_attained} and Proposition \ref{prop:sign_of_mu_lambda}. Lemma \ref{lemma:alpha_to_lambda_1} implies the asymptotic behavior as $\alpha\to \lambda_1(\Omega)^+$, while the results as $\alpha\to +\infty$ follow from Lemmas \ref{rem:blowup} and \ref{lemma:mu_tends_to_zero}. In the latter case, also $\|u\|_\infty\to+\infty$; moreover, if a solution $u$ has Morse index $k$ (with $k$ being either 1 or 2), then \cite[Theorem 3.2]{EspositoPetralla2011} yields that $u$ has $k$ local maxima $P^i$, $i=1,k$, and
\[
u(x)\leq C \left(\frac{\lambda}{\mu}\right)^{1/(p-1)}\sum_{i=1}^ke^{-\gamma\sqrt{\lambda}|x-P^i|},\qquad \forall x\in \Omega.
\]
This shows that $u$ can have at most $2$ spikes .
\end{proof}

\section{Least energy solutions in the ball}\label{section:FocusingRadial}
From now on we will focus on the case
\[
\Omega:=B_1.
\]
To start with, we collect in the following theorem some well known results about uniqueness and nondegeneracy of positive solutions of equation \eqref{eq:stationary_intro} on the ball.
\begin{theorem}[{\cite{GidasNiNirenberg1979,Kwong1989,KwongLi1992,Korman2002,AftalionPacella2003}}]\label{thm:uniqueness_nondegeneracy}
Let $\lambda\in (-\lambda_1(B_1),+\infty)$ and $\mu>0$ be fixed. Then the problem
\[
-\Delta u+\lambda u=\mu u^p \text{ in } B_1,\qquad u=0 \text{ on } \partial B_1,
\]
admits a unique positive solution $u$, which is nondegenerate, radially symmetric, and decreasing with respect to the radial variable $r=|x|$.
\end{theorem}

\begin{proof}
The existence easily follows from the mountain pass lemma.
The radial symmetry and monotonicity of positive solutions is a direct consequence of \cite{GidasNiNirenberg1979}.

The uniqueness in the case $\lambda>0$ was proved by M. K. Kwong in \cite{Kwong1989} for $N\geq 2$. For $\l\in (-\lambda_1(B_1),0)$, the uniqueness in dimension $N\geq3$ was proved by M. K. Kwong and Y. Li \cite[Theorem 2]{KwongLi1992} (see also \cite{Zhang1992}), whereas in dimension $N=2$ it was proved by P. Korman \cite[Theorem 2.2]{Korman2002}. The case $\lambda=0$ is treated in Section 2.8 of \cite{GidasNiNirenberg1979}.

As for the nondegeneracy, for $\lambda>0$ this follows from \cite[Theorem 1.1]{AftalionPacella2003},
since we know that $u$ has Morse index one, as it is a mountain pass solution for $J_{\mu,\lambda}$
(recall that such functional is defined as in \eqref{eq:Jmulambda}).
As for $\lambda\in (-\lambda_1(B_1),0]$, we could not find a precise reference and for this reason
we present here a proof, following some ideas of \cite{KabeyaTanaka1999}.

Assume by contradiction that $u$ is a degenerate solution for some $\lambda\in (-\lambda_1(B_1),0]$. This means that there exists $0\neq w\in H^1_0(B_1)$ solution of
\[
-\Delta w+\lambda w=pu^{p-1}w,
\]
hence $w\in H^1_{0,\text{rad}}(B_1)$ and $J_{\mu,\lambda}''(u)[w,\xi]=0$ for all $\xi\in H^1_0(B_1)$. Moreover, we have that $J_{\mu,\l}''(u)[u,u]=-(p-1)\m \int_{B_{1}} u^{p+1}\,dx<0,$
and thus
\[
J_{\mu,\lambda}''(u)[h,h]\leq 0,\qquad \forall h\in H:=\mathrm{span}\{u,w\}.
\]
For $\delta>0$, consider the perturbed functional
\begin{equation}\label{eq:J_delta_nondegeneracy}
I_\delta(w)=\int_{B_1} \left( \frac{|\nabla w|^2}{2}+\frac{\lambda+\delta u^{p-1}}{2} w^2 -\frac{\mu+\delta}{p+1} (w^+)^{p+1}  \right) \,dx.
\end{equation}
On the one hand, this functional satisfies, for every $h\in H\setminus\{0\}$,
\begin{eqnarray}\label{eq:Morseindex_contradiction}
I_\delta''(u)[h,h]&=& J_{\mu,\lambda}''(u)[h,h] + \int_{B_{1}} (\delta u^{p-1}h^2-p\delta u^{p-1}h^2) \,dx \nonumber \\
& \leq& -(p-1)\delta \int_{B_{1}} u^{p-1}h^2\,dx < 0,
\end{eqnarray}
On the other hand, $I_\delta$ has a mountain pass geometry for $\delta$ sufficiently small, hence it has a critical point of mountain pass type. Every non-zero critical point of $I_\delta$ is positive (by the maximum principle) and it solves
\[
\left\{\begin{array}{ll}
-\Delta w=V_\delta(r)w +(\mu+\delta) w^p \quad & \text{ in }{B_{1}}\\
w>0 & \text{ in }{B_{1}}\\
w\in H^1_0({B_{1}}),
\end{array}\right.
\]
for $V_\delta(r):=-\lambda-\delta u^{p-1}$. Now this problem has a unique radial solution, which is
$u$ itself, which is in contradiction with \eqref{eq:Morseindex_contradiction}. The uniqueness of
this perturbed problem follows from \cite[Theorem 2.2]{Korman2002} in case $\lambda<0$ (in fact,
$V_\delta(r)>0$  and $\frac{d}{dr}[r^{2n(\frac{1}{2}-\frac{1}{p+1})}V_\delta(r)]\geq 0$), while in
case $\lambda=0$ we can reason exactly as in \cite[Proposition 3.1]{FelmerMartinezTanaka2008} (the
proof there is for the annulus, but the argument also works in case of a ball).
\end{proof}

\begin{remark}\label{rem:morse_index_in_B}
As we already mentioned, the Morse index of $u>0$ as a critical point of $J_{\mu,\l}$ is $1$.
Recalling the definition of $I_\delta$ in
\eqref{eq:J_delta_nondegeneracy}, we have that also the Morse index of $I''_\delta(u)$ is $1$,
at least if $\lambda>-\lambda_1(B_1)$ and if $\delta>0$ is small enough. When $\lambda<0$ this was shown in the proof of
the previous result, where we have dealt also with the case $\lambda=0$. The proof for $\lambda>0$
is the same as in the latter case.
\end{remark}
Given $k>N$, as before let us take $X=\{w\in W^{2,k}(B_1):\ w=0 \text{ on } \partial B_1\}$.
Let us introduce the map $F:X\times \R^3\to L^k(B_1)\times \R^2$ defined by
$$
F(u,\mu,\lambda,\alpha)=\left(\Delta u-\lambda u+\mu u^p,\int_{B_1} u^2\, dx-1,
\int_{B_1} |\nabla u|^2\,dx-\alpha\right),
$$
and its null set restricted to positive $u$
\[
\mathcal{S}=\left\{(u,\mu,\lambda,\alpha)\in X\times \R^3: u>0,\ F(u,\mu,\lambda,\alpha)=(0,0,0)\right\}.
\]
It is immediate to check that $\mathcal{S}\cap\{\alpha\leq\lambda_{1}(B_1)\}=
\{(\varphi_1,0,-\lambda_{1}(B_1),\lambda_{1}(B_1)\}$, so that
\[
\mathcal{S}^\pm:=\mathcal{S}\cap\{\pm\mu>0\}\subset\{\alpha>\lambda_{1}(B_1)\}.
\]
We are going to show that $\mathcal{S}^+$ can be parameterized in a smooth way on $\alpha$,
thus proving the part of Theorem \ref{thm:main_final} regarding focusing nonlinearities. As we
mentioned, the (easier) study of $\mathcal{S}^-$ is postponed to Appendix \ref{app:C}. In
view of the application of the Implicit Function Theorem, we have the following.
\begin{lemma}\label{lem:nondegeneracy_in_the_ball}
Let $(u,\mu,\lambda,\alpha) \in \mathcal{S}^+$. Then the linear bounded operator
\[
F_{(u,\mu,\lambda)}(u,\mu,\lambda,\alpha):X\times\R^2 \to L^k(B_1)\times\R^2
\]
is invertible.
\end{lemma}
\begin{proof}
The lemma is a direct consequence of the Fredholm Alternative and of the Closed Graph Theorems,
once we show that the operator above is injective. Let us suppose by contradiction the
existence of $(v,m,l)\neq(0,0,0)$ such that
$F_{(u,\mu,\lambda)}(u,\mu,\lambda,\alpha)[v,m,l]=(0,0,0)$. This explicitly writes
\begin{equation}\label{eq:cond_nondeg}
\begin{array}{lll}
-\Delta u+\lambda u= \mu u^p, & \int_{B_1} u^2\,dx=1, & \int_{B_1} |\nabla u|^2\, dx=\alpha, \\
-\Delta v +\lambda v+l u =p\mu u^{p-1} v +m u^p, & \int_{B_1} uv\, dx=0, &
\int_{B_1} \nabla u \cdot \nabla v\, dx=0.
\end{array}
\end{equation}
By testing the two differential equations by $v$ we obtain
\begin{equation}\label{eq:cond_nondeg2}
\int_{B_1} u^p v\,dx=0,\qquad
\int_{B_1} |\nabla v|^2\,dx + \lambda \int_{B_1} v^2 \,dx =p\mu \int_{B_1} u^{p-1} v^2 \,dx,
\end{equation}
so that
\[
J_{\mu,\lambda}''(u)[u,u]<0, \quad J_{\mu,\lambda}''(u)[u,v]=0, \quad J_{\mu,\lambda}''(u)[v,v]=0.
\]
This implies that $J_{\mu,\lambda}''(u)[h,h]\leq 0$ for every $h\in H=\spann\{u,v\}$.
By defining $I_\delta$ as in \eqref{eq:J_delta_nondegeneracy}, for $\delta>0$ small, we obtain $I_{\delta}''(u)[h,h]< 0$
for every $0\neq h\in H$. Since $H$ has dimension 2 ($v=cu$ would imply
$c\int_\Omega u^2 =0$), this contradicts Remark
\ref{rem:morse_index_in_B}.
\end{proof}

\begin{proposition}\label{prop:focusing_radial_regular_curve}
$\mathcal{S}^+$ is a smooth curve, parameterized by a map
\[
\alpha\mapsto (u(\alpha),\mu(\a),\l(\a)), \qquad \alpha\in (\lambda_1({B_{1}}),+\infty).
\]
In particular, $u(\alpha)$ is the unique maximizer of $M_\alpha$ (as defined in
\eqref{eq:M_intro}).
\end{proposition}
\begin{proof}
To start with, Lemma \ref{lemma:m_M_attained} and Proposition \ref{prop:sign_of_mu_lambda} imply that, for every fixed $\alpha^*>
\lambda_1({B_{1}})$, there exists
at least a corresponding point in $\mathcal{S}^+$. If $(u^\ast,\mu^\ast,\lambda^\ast,\a^\ast)$
denotes any of such points (not necessarily related to $M_{\a^*}$), then by Lemma
\ref{lem:nondegeneracy_in_the_ball} it can be continued, by means of the Implicit Function
Theorem, to an arc $(u(\alpha),\mu(\alpha),\lambda(\a))$, defined on a maximal interval
$(\underline{\alpha},\overline{\alpha})\ni\alpha^\ast$, chosen in such a way that $\mu(\a)>0$
on such interval.
Since $u(\alpha)$ solves the equation,
standard arguments involving the Maximum Principle and Hopf Lemma allow to obtain that
$u(\alpha)>0$ (recall that we are using the $W^{2,k}$-topology) along the arc, which consequently
belongs to $\mathcal{S}^+$. We want to show that
$(\underline{\alpha},\overline{\alpha})=(\lambda_1({B_{1}}),+\infty)$

Let us assume by contradiction $\underline{\alpha}>\lambda_1(\Omega)$.
For $\alpha_n\to \underline{\alpha}^+$, Lemma \ref{lemma:case_alpha_n_bounded} implies that,
up to a subsequence,
$$
u_n\rightharpoonup \bar u \text{ in }H^1_0(\Omega),\ \lambda_n\to \bar\lambda,\ \mu_n\to \bar \mu.
$$
Thus
$$
-\Delta \bar u+\bar \lambda \bar u=\bar \mu \bar u^p \qquad \text{ in } \Omega,
$$
and the convergence $u_n\to \bar u$ is actually strong in $H^2(\Omega)$. Then $\int_\Omega
|\nabla \bar u|^2\, dx=\underline{\alpha}>\lambda_1(\Omega)$, so that $\bar \mu>0$.
Thus Lemma \ref{lem:nondegeneracy_in_the_ball} allows to reach a contradiction with the
maximality of $\underline{\alpha}$, and therefore $\underline{\alpha}=\lambda_1(\Omega)$.
Analogously, we can show that $\overline{\a}=+\infty$.

Once we know that $\mathcal{S}^+$ is the disjoint union of smooth curves,
each one parameterized by $\alpha\in (\lambda_1({B_{1}}),+\infty)$, it only remains to show that
the curve of solutions is indeed unique. Suppose by contradiction that, for $\alpha_n\to
\lambda_1({B_{1}})$, there exist $(u_1(\alpha_n),\mu_1(\alpha_n),\lambda_1(\alpha_n))\neq
(u_2(\alpha_n),\mu_2(\alpha_n),\lambda_2(\alpha_n))$ for every $n$. Then by Lemma
\ref{lemma:alpha_to_lambda_1} both triplets converge to $(\vphi_1,0,-\lambda_1({B_{1}}))$, in
contradiction with Proposition \ref{prop:neighborhood of lambda_1}.
\end{proof}

\begin{corollary}\label{coro:derivate}
Writing
\[
\frac{d}{d\alpha}(u(\alpha),\mu(\a),\l(\a)) = (v(\alpha),\mu'(\a),\l'(\a)),
\]
we have
\begin{equation*}
-\Delta v+\lambda' u+\lambda v=p\mu u^{p-1} v+\mu' u^p, \qquad v\in H^1_0({B_1}),
\end{equation*}
\begin{equation}\label{eq:conditions_on_u'}
\int_{B_1} u v\, dx=0, \qquad  \int_{B_1} \nabla u\cdot \nabla v\, dx=\frac{1}{2},
\end{equation}
and
\begin{equation}\label{eq:mu_u^p_v}
\mu \int_{B_1} u^pv\, dx=\frac{1}{2},
\qquad
\m'\int_{B_1} u^{p+1}\,dx=\l'-\frac{p-1}{2}.
\end{equation}
\end{corollary}
\begin{proof}
Direct computations (by differentiating $F(u(\alpha),\mu(\a),\l(\a),\a)=0$ and testing
the differential equations by $u$ and $v$).
\end{proof}
In the following, we address the study of the monotonicity properties of the map
\[
\alpha\mapsto (u(\alpha),\mu(\a),\l(\a))
\]
introduced above, $v$ always denoting the derivative of $u$ with respect to $\a$.
\begin{lemma}\label{lemma:lambda'<0}
$\lambda'(\alpha)>0$ for every $\alpha>\lambda_1(B_1)$.
\end{lemma}
\begin{proof}
Let $(h,k)\in\R^2$, and let us consider the quadratic form
\[
J''_{\mu,\lambda}(u)[hu+kv,hu+kv] =: a h^2 + 2b hk + c k^2.
\]
Using Corollary \ref{coro:derivate} we obtain
\[
\begin{split}
a &= J''_{\mu,\lambda}(u)[u,u] = \int_{B_1}\left[|\nabla u|^2 + \lambda u^2-p\mu u^{p+1} \right]\,dx = -(p-1)\mu\int_{B_1} u^{p+1}\,dx\\
b &= J''_{\mu,\lambda}(u)[u,v] = \int_{B_1}\left[\nabla u\cdot\nabla v + \lambda uv-p\mu u^pv
\right]\,dx = -\frac{p-1}{2}\\
c &= J''_{\mu,\lambda}(u)[v,v] = \int_{B_1}\left[|\nabla v|^2 + \lambda v^2-p\mu u^{p-1} v^2 \right]\,dx = \frac{\mu'}{2\mu}.
\end{split}
\]
Since $J''_{\mu,\lambda}(u)$ has (large) Morse index equal to one (Remark
\ref{rem:morse_index_in_B}), and $a<0$,
we have that $b^2-ac>0$, i.e.
\begin{equation*}
\mu'\int_{B_1} u^{p+1}\,dx>-\frac{p-1}{2}.
\end{equation*}
The lemma follows by comparing with equation \eqref{eq:mu_u^p_v}.
\end{proof}

\begin{lemma}\label{lemma:mu'}
If $\omega_N=|\partial B_1|$ then
\begin{equation*}
\mu'\int_{B_1} u^{p+1}\,dx=
\frac{p+1}{2(p-1)}\left[\left(-p+1 + \frac{4}{N}\right) - \frac{4\omega_N}{N}u_r(1)v_r(1)\right].
\end{equation*}
\end{lemma}
\begin{proof}
Recall that both $u$ and $v$ are radial. Since $\int_{B_1}u^2\, dx=1$,
the standard Pohozaev identity writes
\begin{equation*}
\left(\frac{N}{2}-1\right)\int_{B_1}|\nabla u|^2\,dx+\frac{1}{2}\int_{\partial B_1} |\nabla u|^2(x\cdot\nu)\,d\sigma+\frac{\l N}{2} =\frac{\mu N}{p+1}\int_{B_1} |u|^{p+1}\,dx.
\end{equation*}
Inserting the information that $u$ is radial and the equalities $\a=\int_{B_1} |\nabla u|^2\,dx$,
$\a+\l=\m \int_{B_1} u^{p+1}\,dx$, we obtain
\begin{equation*}
\lambda=\frac{2}{N}\frac{p+1}{p-1}\a -\a -\frac{\omega_N}{N}\frac{p+1}{p-1}u_r(1)^2.
\end{equation*}
Differentiating with respect to $\a$ we have
\[
\l'=\frac{2}{N}\frac{p+1}{p-1} -1 -\frac{2\omega_N}{N}\frac{p+1}{p-1}u_r(1)v_r(1).
\]
The result follows by recalling relation \eqref{eq:mu_u^p_v}.
\end{proof}
The following crucial lemma shows that, if $p$ is subcritical or critical, then $\mu$ is
an increasing function of $\a$.
\begin{lemma} \label{lem:mu'>0_criticalcase}
If $p\leq 1+4/N$ then $\mu'(\alpha)>0$ for every $\alpha>\lambda_1(B_1)$.
\end{lemma}
\begin{proof}
The proof goes by contradiction: suppose that $\mu'(\bar \alpha)\leq 0$ for some $\bar\alpha>\lambda_1(B_1)$. In the remaining of the proof all quantities are evaluated at such $\bar\alpha$.

\paragraph{Step 1.} Let $v:=\frac{d}{d\alpha} u|_{\a=\bar\a}$, then $v_r(1)<0$ in case $p< 1+4/N$
and $v_r(1)\leq0$ if $p= 1+4/N$. This is an immediate consequence of Lemma \ref{lemma:mu'}, being
$u_r(1)<0$ by Hopf Lemma.

\paragraph{Step 2.} We claim that, if $r$ is sufficiently close to $1^-$, then $v(r)>0$. Since $v(1)=0$, this is obvious if $v_r(1)<0$. Hence it only remains to consider the case $p= 1+4/N$ and $v_r(1)=0$.

From the equation for $v$ written in the radial coordinate:
\begin{equation*}
-v_{rr}-\frac{N-1}{r} v_r+\lambda v + \lambda' u=p\mu u^{p-1} v +\mu' u^p,\quad r\in (0,1)
\end{equation*}
we know (by letting $r\to1^-$) that $v_{rr}(1)=0$. Differentiating both sides of the above equation, we can write
\begin{multline*}
-v_{rrr}+\frac{N-1}{r^2}v_{r}-\frac{N-1}{r}v_{rr}+\lambda v_r+\lambda'u_r\\=  p(p-1) \mu u^{p-2} u_r v + p \mu u^{p-1} v_r+p\mu'u^{p-1}u_r;
\end{multline*}
now, if $p\geq2$, the limit as $r\to1^-$ yields
$$
-v_{rrr}(1)+\lambda' u_r(1)=0.
$$
On the other hand, if $p<2$, the same identity holds, since by the l'H\^opital's rule
$$
\lim_{r\to 1^-} u^{p-2}u_rv=\lim_{r\to 1^-} \frac{u_{rr} v+u_rv_r}{(2-p)u^{1-p}u_r}=\frac{u_{rr}(1)v(1)+u_r(1)v_r(1)}{(2-p)u_r(1)}u(1)^{p-1}=0.
$$
Thus $v_{rrr}(1)<0$ by Lemma \ref{lemma:lambda'<0}, and the claim follows.

\paragraph{Step 3.} Let $\bar r:= \inf\left\{r:v>0\text{ in }(r,1)\right\}$ ($\bar r>0$ since
$\int_{B_1} uv \,dx = 0$). We claim that $v\leq 0$ in $B_{\bar r}$. If not,
there would be $0\leq r_1<r_2 \leq
\bar r$ with the property that $v>0$ in $(r_1,r_2)$ and $r_iv(r_i)=0$. Defining
\[
v_1 := v|_{B_{r_2}\setminus B_{r_1}},\qquad v_2 := v|_{B_{1}\setminus B_{\bar r}}
\]
we have that $v_i\in H^1_0(B_1)$, $v_i\geq 0$ for $i=1,2$, and $v_1,v_2$ are linearly independent. One can use
the equation for $v$ in order to evaluate
\[
J''_{\mu,\l}(u)[v,v_i] = \int_{B_1} (\nabla v\cdot \nabla v_i + (\l  -p\mu u^{p-1})v v_i   )\,dx=
\int_{B_1}(\mu' u^p v_i - \l' u v_i)\,dx < 0
\]
and obtain
\[
J''_{\mu,\l}(u)[t_1v_1+t_2v_2,t_1v_1+t_2v_2]<0\quad\text{ whenever }t_1^2 + t_2^2\neq0,
\]
in contradiction with the fact that the Morse index of $u$ is $1$ (Remark
\ref{rem:morse_index_in_B}).

\paragraph{Step 4.} Once we know that $v\leq 0$ in $B_{\bar r}$ and that $v>0$ in $B_1\setminus
B_{\bar r}$, we can combine the first equations in \eqref{eq:conditions_on_u'} and \eqref{eq:mu_u^p_v},
together with the fact that $u$ is monotone decreasing with respect to $r$, to write
\[
\begin{split}
\frac{1}{2\mu} &= \int_{B_1} u^p v\,dx = \int_{B_1\setminus B_{\bar r}} u^pv\,dx + \int_{B_{\bar r}} u^pv\,dx\\
   &\leq (\max_{B_1\setminus B_{\bar r}} u^{p-1}) \int_{B_1\setminus B_{\bar r}} uv \,dx+
    (\min_{B_{\bar r}} u^{p-1}) \int_{B_{\bar r}} uv\,dx\\
   &= u^{p-1}(\bar r) \int_{B_1\setminus B_{\bar r}} uv\,dx +
    u^{p-1}(\bar r) \int_{B_{\bar r}} uv \,dx= 0,
\end{split}
\]
a contradiction.
\end{proof}
\begin{remark}
When $1+4/N<p<2^*-1$, Lemma \ref{lemma:mu_tends_to_zero} implies that $\mu(+\infty)=0$. Since
also $\mu(\l_1(B_1)^+)=0$, we deduce that $\mu'$ must change sign in the supercritical regime.
Numerical experiments suggest that this should happen only once, so that
$\mu$ should have a unique global maximum and be strictly monotone elsewhere,
see Remark \ref{rem:simulation} ahead.
\end{remark}
We are ready to prove the existence of least energy solutions for equation
\eqref{eq:stationary_intro}.
\begin{proof}[Proof of Theorem \ref{thm:intro_existence}]
Recalling Definition \ref{defi:les}, let $\rho>0$ be fixed, and let $U\in\mathcal{P}_\rho$. Then
\[
\int_{B_1}U^2\,dx=\rho, \ U>0, \quad\text{ and }-\Delta U+ \lambda U=U^p
\]
for some $\lambda$. Then, setting
\(
u = \rho^{-1/2}U,
\)
direct calculations yield
\[
\int_{B_1}u^2\,dx=1, \ u>0, \quad\text{ and }-\Delta u+ \lambda u=\rho^{(p-1)/2}u^p.
\]
Writing $\int_{B_1}|\nabla u|^2\,dx=\alpha$, this amounts to say that
\(
(u,\rho^{(p-1)/2},\lambda,\alpha)\in\mathcal{S}^+.
\)
Equivalently,
\[
U\in\mathcal{P}_\rho \iff \rho = \mu^{2/(p-1)}\text{ and }U=\mu^{1/(p-1)}u\text{
for some }(u,\mu,\l,\a)\in\mathcal{S}^+.
\]
We divide the end of the proof in three cases.

\paragraph{Case 1: $1<p<1+4/N$.} By Lemmas \ref{lemma:mu_tends_to_zero},
\ref{lem:mu'>0_criticalcase} and Proposition \ref{prop:focusing_radial_regular_curve} we have
that, for every $\rho$, there exists exactly one point in $\mathcal{S}^+$ satisfying
$\mu^{2/(p-1)}=\rho$.

\paragraph{Case 2: $p=1+4/N$.} The same as the previous case, taking into account that, by
Lemma \ref{lemma:mu_tends_to_zero}, $\mathcal{P}_{\mu^{2/(p-1)}}$ is not empty if and only
if $\mu<\|\Kwong\|_{L^2(\R^N)}^{p-1}$.

\paragraph{Case 3: $1+4/N<p<2^*-1$.} Since in this case $\mu(\lambda_1(B_1))=\mu(+\infty)=0$
(by Lemma \ref{lemma:mu_tends_to_zero}), then
\[
\mu^*=\max_{(\lambda_1(B_1),+\infty)}\mu
\]
is well defined and achieved. Furthermore, $\mathcal{P}_{\mu^{2/(p-1)}}$ is empty for
$\mu>\mu^*$, and it contains at least two points for $0<\mu<\mu^*$. It remains to prove that, if
$0<\rho\leq\rho^*=(\mu^*)^{(p-1)/2}$, then $\gslev_{\rho}$ is achieved. This is immediate whenever
$\mathcal{P}_\rho$ is finite. Otherwise, let $u_n=u(\alpha_n)$, with $\mu(\alpha_n)=\rho^{(p-1)/2}$,
denote a minimizing sequence. Then Lemma \ref{lemma:mu_tends_to_zero} implies that $\alpha_n$ is
bounded, and by continuity the same is true for $\lambda_n$. We deduce that, up to subsequences,
$u_n \to u^*\in\gsset_{\bar\mu}$, and $J_{\bar\mu,0}(u^*)=\gslev_{\rho}$.
\end{proof}
\begin{remark}\label{rem:segno_scontato}
By comparing Theorem \ref{thm:intro_existence} and Proposition \ref{prop:fibichmerle}, we have
that when $p\leq1+4/N$ and positive least energy solutions exist, the condition $U>0$ may be safely
removed from Definition \ref{defi:les} without altering the problem (in fact, also the condition
$-\Delta U + \l U =U^{p+1}$ for some $\lambda$ is not necessary). On the other hand, in other cases
it is essential. For instance, when $p$ is critical then the set of non necessarily positive
solutions with fixed mass
\[
\mathcal{P}_\rho'=\left\{ U\in H^1_0(B_1):\ \mathcal{Q}(U)=\rho,  \ \exists\lambda:\ -\Delta U+
\lambda U=U^p \right\}
\]
is not empty also when $\rho\geq \|\Kwong\|_{L^2(\R^N)}^{2}$, as illustrated in
\cite[Figure 1]{FibichMerle2001}.
\end{remark}

\section{Stability results}\label{section:Stability}

In this section we discuss orbital stability of standing wave solutions
$e^{i\lambda t}U(x)$ for the NLS \eqref{eq:Schrodinger_intro}. We recall
that such solutions are called orbitally stable if for each $\eps>0$ there exists $\delta>0$ such
that whenever $\Phi_0\in H^1_0(B_1,\C)$ is such that $\|\Phi_0-U\|_{H^1_0(B_1,\C)}<\delta$ and
$\Phi(t,x)$ is the solution of \eqref{eq:Schrodinger_intro} with $\Phi(0,\cdot)=\Phi_0$ in some
interval $[0,t_0)$, then $\Phi(t,\cdot)$ can be continued to a solution in $0\leq t<\infty$ and
\[
\sup_{0<t<\infty}\inf_{s\in \R} \|\Phi(t,\cdot)-e^{-i\lambda s}U\|_{H^1_0(B_1,\C)}<\eps;
\]
otherwise, they are called unstable.
To do this, we lean on the following result by Fukuizumi, Selem and Kikuchi, which expresses
in our context the abstract theory developed in \cite{GrillakisShatahStrauss}.
\begin{proposition}[{\cite[Proposition 5]{Fukuizumi2012}}]
Let us assume local existence as in Theorems \ref{thm:intro_stability} and \ref{thm:intro_stab_2},
and let $R_\lambda$ be the unique positive solution of \eqref{eq:stationary_intro}.
\begin{itemize}
 \item If $\partial_\lambda\|R_\lambda\|_{L^2}^2>0$ then $e^{i\lambda t}R_\lambda$ is orbitally
 stable;
 \item if $\partial_\lambda\|R_\lambda\|_{L^2}^2<0$ then $e^{i\lambda t}R_\lambda$ is unstable.
\end{itemize}
\end{proposition}
\begin{corollary}\label{coro:mu'>0=>stab}
Let $(u(\a),\mu(\a),\lambda(\a),\alpha)\in\mathcal{S}^+$, with $U(\a)=\mu^{1/(p-1)}(\a)u(\a)$
denoting the corresponding solution of \eqref{eq:stationary_intro} (with $\l=\l(\a)$). Then
\begin{itemize}
 \item if $\mu'(\a)>0$ then $e^{i\lambda(\a) t}U(\a)$ is orbitally stable;
 \item if $\mu'(\a)<0$ then $e^{i\lambda(\a) t}U(\a)$ is unstable.
\end{itemize}
\end{corollary}
\begin{proof}
Taking into account Proposition \ref{prop:focusing_radial_regular_curve} and Lemma
\ref{lemma:lambda'<0}, and reasoning as in the proof of Theorem \ref{thm:intro_existence},
we have that $R_{\l(\a)}=\mu^{1/(p-1)}(\a)u(\a)$,{ so that }
\[
\partial_\lambda\|R_\lambda\|_{L^2}^2
=\frac{\left(\mu^{2/(p-1)}\right)'(\a)}{\l'(\a)}=\frac{2\mu^{(3-p)/(p-1)}(\a) }{(p-1)\l'(\a)}
\mu'(\a).\qedhere
\]
\end{proof}
We recall that $\mu'$ may be negative only when $p$ is supercritical. Such case is enlightened by
the following lemma.
\begin{lemma}\label{lemma:gs=>mu'>0}
Let $p>1+4/N$, and consider the map $\alpha\mapsto(u(\alpha),\mu(\alpha),\lambda(\alpha))$
defined as in Proposition \ref{prop:focusing_radial_regular_curve}. If $\alpha_1<\alpha_2$ are such that
\[
\mu(\alpha)>\mu(\alpha_1)=\mu(\alpha_2)=:\bar\mu\quad\text{ for every }\alpha\in(\alpha_1,\alpha_2)
\]
then
\[
J_{\bar\mu,0}(u(\alpha_1))<J_{\bar\mu,0}(u(\alpha_2)).
\]
\end{lemma}
\begin{proof}
Writing $M(\alpha) = M_\alpha = \int_{{B_{1}}} u^4(\alpha)\,dx$, we have that
\[
2 J_{\bar\mu,0}(\alpha_i) = \alpha_i - \frac{\bar\mu}{2} M(\alpha_i).
\]
Now, equation \eqref{eq:mu_u^p_v} yields $M'(\alpha) = 4\int_{B_{1}} u^3 v \,dx = 2/\mu(\alpha)$,
where as usual $v:=\frac{d}{d\alpha} u$. Lagrange Theorem applied to $M$ forces the existence of
$\alpha^*\in(\alpha_1,\alpha_2)$ such that
\[
\frac{M(\alpha_2)-M(\alpha_1)}{\alpha_2-\alpha_1} = \frac{2}{\mu(\alpha^*)} < \frac{2}{\bar\mu},
\]
which is equivalent to the desired statement.
\end{proof}
We are ready to give the proofs of our stability results.
\begin{proof}[Proof of Theorems \ref{thm:intro_stability}, \ref{thm:intro_stab_2}]
The proof in the subcritical and critical case is a direct consequence of Lemma
\ref{lem:mu'>0_criticalcase} and Corollary \ref{coro:mu'>0=>stab} (recall that in such
case there is a full correspondence between least energy solutions and least action ones).
To show Theorem \ref{thm:intro_stability}, point 2., we prove stability for any $\rho>0$ such that
$\bar\mu = \rho^{(p-1)/2}$ is a regular value of the map $\alpha\mapsto\mu(\a)$, the conclusion
following by Sard Lemma. Recalling that $\mu(\lambda_1(B_1))=\mu(+\infty)=0$, we have that, if
$\bar\mu$ is regular, then its counterimage $\{\a:\mu(\a)=\bar\mu\}$ is the union of a finite
number of pairs $\{\a_{i,1},\a_{i,2}\}$, each of which satisfies the assumptions of Lemma
\ref{lemma:gs=>mu'>0}, and moreover $\mu'(\a_{i,1})>0>\mu'(\a_{i,2})$. Since such counterimage
is in 1-to-1 correspondence with $\mathcal{P}_\rho$, and
\[
\mathcal{E}(U(\a_{i,j}))=\mathcal{E}(\bar\mu^{1/(p-1)}u(\a_{i,j}))=\bar\mu^{2/(p-1)}J_{\bar\mu,0}(u(\a_{i,j})),
\]
we deduce from Lemma \ref{lemma:gs=>mu'>0} that the least energy solution corresponds to
$\a_{i,1}$, for some $i$, and the conclusion follows again by Corollary \ref{coro:mu'>0=>stab}.
\end{proof}
\begin{remark} \label{rem:simulation}
In the supercritical case $p>1+4/N$, we expect orbital stability for every $\rho\in (0,\rho^\ast)$,
and instability for $\rho=\rho^*$. Indeed, in case $N=3$, $p=3$, we have plotted numerically the
graph of $\mu(\alpha)$ in Figure \ref{fig:mu}.
\begin{figure}[h]
\begin{center}
\includegraphics[scale=0.8]{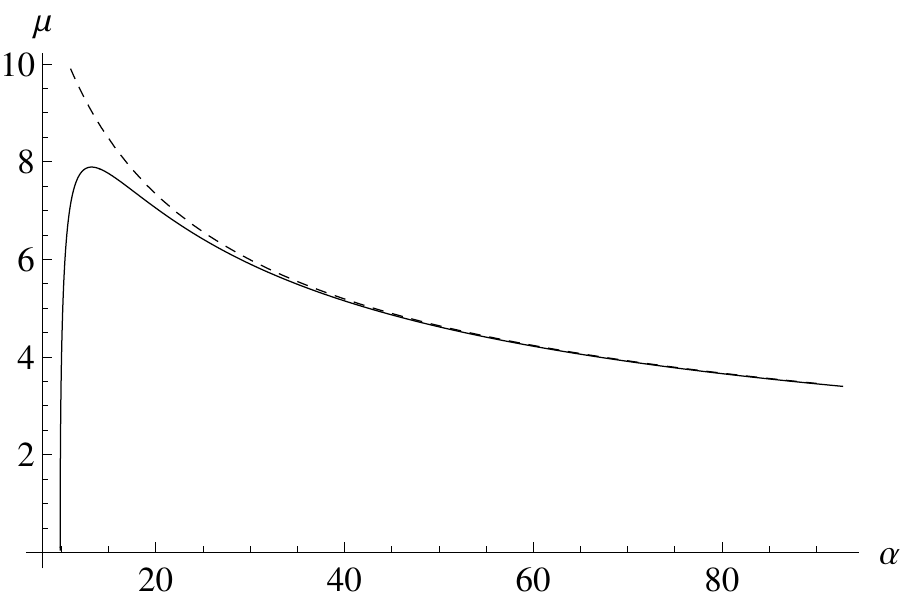}
\caption{numerical graph of $\alpha\mapsto\mu(\a)$ in the supercritical case $N=3$, $p=3$
(continuous line) and of the map
$\a\mapsto \a^{-1/2}\cdot\sqrt{3}\int_{\R^3}Z^2_{3,3}\,dx$ (dashed line). The latter is the theoretical asymptotic expansion of
$\mu(\a)$ as $\alpha\to+\infty$, as predicted by Lemmas \ref{rem:blowup},
\ref{lemma:mu_tends_to_zero}.}
\label{fig:mu}
\end{center}
\end{figure}\break
The picture suggests that $\mu$ has a unique local maximum $\mu^*$, associated to the maximal value
of the mass
$\rho^*=(\mu^*)^{(p-1)/2}$. For any $\mu<\mu^*$, we have exactly two solutions, and the least energy
one corresponds to $\mu'(\alpha)>0$, hence it is associated with an orbitally stable standing wave.
For $\mu=\mu^*$ we have exactly one solution; in such case the abstract theory developed in
\cite{GrillakisShatahStrauss} predicts the corresponding
standing wave to be unstable.
\end{remark}


\appendix

\section{Gagliardo-Nirenberg inequalities}\label{app:A}
It is proved in \cite{Weinstein1983} that the following sharp Gagliardo-Nirenberg
inequality holds for every $u\in H^1_0(\Omega)$
\begin{equation}\label{eq:gagliardo_nirenberg}
\|u\|_{L^{p+1}(\R^N)}^{p+1}\leq C_{N,p}\|u\|_{L^2(\R^N)}^{p+1-N(p-1)/2}
\|\nabla u\|_{L^2(\R^N)}^{N(p-1)/2},
\end{equation}
and that the best constant $C_{N,p}$ is achieved by (any rescaling of) $\Kwong$.

When dealing with $H^1_0(\Omega)$, $\Omega\neq\R^N$, one can prove that the identity
holds with the same best constant: in fact, one inequality is trivial, and the other
is obtained by constructing a suitable competitor of the form $u(x)=(h\Kwong(kx)-j)^+$,
for suitable $h,k,j$, and exploiting the exponential decay of $Z$. Contrarily to the previous
case, now such constant can not be achieved, otherwise we would contradict \cite{Weinstein1983}.
This is related with the maximization problem \eqref{eq:M_intro}, since
\[
C_{N,p} = \sup_{H^1_0(\Omega)\setminus\{0\}} \frac{\|u\|_{L^{p+1}(\Omega)}^{p+1}}{\|u\|_{L^2(\Omega)}^{p+1-N(p-1)/2}
\|\nabla u\|_{L^2(\Omega)}^{N(p-1)/2}}= \sup_{\alpha\geq\lambda_1(\Omega)}\frac{M_\a}{\a^{N(p-1)/2}}.
\]
By the above considerations we deduce that
\begin{equation}\label{eq:ganiM}
M_\a<C_{N,p}\a^{N(p-1)/2}\text{ for every }\a,\qquad \lim_{\alpha\to+\infty}
\frac{M_\a}{\a^{N(p-1)/2}}=C_{N,p}.
\end{equation}
For the readers convenience, we deduce the following well known result.
\begin{proposition}\label{prop:fibichmerle}
Let $\rho>0$ be fixed. The infimum
\[
\inf \left\{\mathcal{E}(U):\ U\in H^1_0(\Omega) \text{ and } \mathcal{Q}(U)=\rho \right\}
\]
\begin{itemize}
\item[(i)] is achieved by a positive function if either $1<p<1+4/N$ or $p=1+4/N$ and
$\rho < \|\Kwong\|_{L^2(\R^N)}^{2}$
;
\item[(ii)] equals $-\infty$ if either $1+4/N<p<2^*-1$ or $p=1+4/N$ and
$\rho > \|\Kwong\|_{L^2(\R^N)}^{2}$.
\end{itemize}
\end{proposition}
\begin{proof}
As usual, writing $u = \rho^{-1/2}U$ and $\bar\mu=\rho^{(p-1)/2}$, we have that the above minimization problem is equivalent to
\[
\inf \left\{J_{\bar\mu,0}(u):\ u\in H^1_0(\Omega) \text{ and } \|u\|_{L^2(\Omega)}=1 \right\},
\]
where $J_{\mu,\l}$ is defined in \eqref{eq:Jmulambda}. In turn, this problem can be written as
\[
\inf_{\alpha\geq\lambda_1(\Omega)} \frac12\a - \frac{\bar\mu}{p+1}{M_\a}.
\]
The lemma follows from equation \eqref{eq:ganiM}, recalling that, when $p=1+4/N$,
\[
C_{N,p} = \left(1+\frac2N\right)\left(\int_{\R^N} \Kwong^2\,dx\right)^{-2/N}
\]
by Pohozaev identity.
\end{proof}

\section{Proof of Remark \ref{rem:fisitu}}\label{app:B}

Given $k>N$, consider $X=\{w\in W^{2,k}(\Omega):\ w=0 \text{ on } \partial \Omega\}$, $Y=L^k(\Omega)$ and $U=\{w\in X:\ w>0 \text{ in } \Omega \text{ and } \partial_\nu w<0 \text{ on } \partial \Omega\}$. The continuous embedding $W^{2,k}(\Omega)\hookrightarrow C^{1,\gamma}(\overline\Omega)$ implies that $U$ is an open subset of $X$. We want to show that the map $\Phi:X\times \R^2\to Y\times \R^2$, defined in Section \ref{section:Ambrosetti-Prodi} by
$$
\Phi(u,\mu,\lambda)=\left(\Delta u-\lambda u+\mu u^p,\int_\Omega u^2\, dx-1,\int_\Omega |\nabla u|^2\, dx\right),
$$
is of class $C^2$ for every $1<p<2^\ast-1$. For $p\geq 2$ this is clear, and the aim of this appendix is to give a proof the regularity for $1<p<2$. We will use the ideas of the proof of \cite[Lemma 4.1]{OrtegaVerzini2004}.
\begin{lemma}\label{lemma:auxiliary_for_p<2}
\begin{itemize}
\item[(i)] There exists $k_1>0$ such that
$$
|u|\leq k_1 \|u\|_X \varphi_1\qquad \forall u\in X.
$$
\item[(ii)] Given $u\in U$ there exists $\eps,k_2>0$ such that
$$
\|v-u\|_X\leq \eps \Rightarrow v\geq k_2\varphi_1.
$$
\end{itemize}
\end{lemma}
\begin{proof}
(i) If the conclusion does not hold, then there exists $u_n\in X$ with $\|u_n\|_X=1$ and $x_n\in \Omega$ with $x_n\to x^\ast$ such that $|u_n(x_n)|>n\varphi_1(x_n)$. Dividing both sides of the previous inequality by $\textrm{d}(x_n,\partial \Omega)=|x_n-x_n^\ast|$, we get
$$
|-\partial_\nu u_n(x_n^\ast)|\geq \frac{n}{2}(-\partial_\nu \varphi_1(x_n^\ast))\to +\infty \qquad \text{ as }n\to \infty,
$$
which contradicts the uniform bound $\|u_n\|_{C^{1,\gamma}(\overline \Omega)}\leq C \|u_n\|_X=C$.

(ii) This is a direct consequence of the continuous embedding $W^{2,k}(\Omega)\hookrightarrow C^{1,\gamma}(\overline\Omega)$.
\end{proof}

We are now ready to prove the following result, which implies that $\Phi$ is indeed $C^2$.
\begin{lemma}
Let $1<p<2$. Then the map $N:U\to Y$ defined by $N(u)=u^{p-1}$ is of class $C^1$, with
$$
N'(u)\psi=(p-1)u^{p-2}\psi \qquad \forall u\in U,\ \psi\in X.
$$
\end{lemma}
\begin{proof}
1. Given $u\in U$, define $L(u):X\to Y$ by $L(u)\psi=(p-1)u^{p-2}\psi$. We claim that $L(u)\in \Lcal(X,Y)$ and that the map
$$
U\to \Lcal(X,Y),\qquad u\mapsto L(u)
$$
is continuous. The first statement is a consequence of Lemma \ref{lemma:auxiliary_for_p<2}-(i):
$$
|L(u)\psi|=(p-1)u^{p-2}|\psi|\leq c_1\varphi_1^{p-2}\|\psi\|_X \varphi_1=c_1\varphi_1^{p-1}\|\psi\|_X\leq c_2\|\psi\|_X.
$$

As for the continuity, take $u_n,u\in U$ such that $u_n\to u$ in $X$. We have
$$
u_n^{p-2}(x)\to u^{p-2}(x)\qquad \text{ for every } x\in \Omega,
$$
and, by Lemma \ref{lemma:auxiliary_for_p<2}-(ii),
$$
|(u_n^{p-2}-u^{p-2})|\varphi_1\leq 2 k_2^{p-2}\varphi_1^{p-2}\varphi_1\leq C.
$$
for $n$ sufficiently large. Hence
$$
\int_\Omega |(u_n^{p-2}-u^{p-2})|^p\varphi_1^p\, dx\to 0 \qquad \text{ as } n\to \infty
$$
by the Lebesgue's dominated convergence theorem, and
\[
\begin{split}
\|\Lcal(u_n)-\Lcal(u)\|_{\Lcal(X,Y)}^p&=\sup_{\|\psi\|_X\leq 1} \| \Lcal(u_n)\psi-\Lcal(u)\psi \|^p_{Y}\\
							&=\sup_{\|\psi\|_X\leq 1} \int_\Omega  |(u_n^{p-2}-u^{p-2})|^p|\psi|^p\, dx\\
							&\leq k_1^p \int_\Omega |(u_n^{p-2}-u^{p-2})|^p\varphi_1^p\, dx\to 0.
\end{split}
\]

2. Finally, let us prove that
\[
\left\|\frac{1}{\eps}((u+\eps \psi)^{p-1}-u^{p-1})-(p-2)u^{p-2}\psi\right\|_Y\to 0
\]
First of all, it is clear that, as $\eps\to 0$,
\[
\frac{1}{\eps}((u+\eps \psi)^{p-1}-u^{p-1})\to (p-2)u^{p-2}\psi \quad \text{pointwise in } \Omega.
\]
(since $u>0$ in $\Omega$). From now on take $\eps$ small so that $u+\eps\psi\geq u/2$ in $\Omega$. Then for each $x\in \Omega$ fixed, from the mean value theorem there exists $|\eps_x|<|\eps|$ so that
\[
\frac{1}{\eps}((u+\eps \psi)^{p-1}-u^{p-1})=(p-1)|(u+\eps_x\psi)^{p-2}\psi|\leq 2^{2-p}(p-1)|u^{p-2}\psi|
\]
and the conclusion follows once again by the Lebesgue's dominated convergence theorem.
\end{proof}
%

\section{The defocusing case $\mu<0$}\label{app:C}

In such case it is not necessary to restrict to spherical domains, therefore in this appendix we
consider a generic smooth, bounded domain $\Omega$. As in Section \ref{section:FocusingRadial} we
work in the space $X=\{w\in W^{2,k}(\Omega):\ w=0 \text{ on } \partial \Omega\}$, for some $k>N$,
and with the map $F:X\times \R^2\to L^k(\Omega)\times \R^2$ defined by
$$
F(u,\mu,\lambda,\alpha)=\left(\Delta u-\lambda u+\mu u^p,\int_\Omega u^2\, dx-1, \int_\Omega |\nabla u|^2-\alpha\right).
$$
We aim at providing a full description of the set
\[
\mathcal{S}^-=\left\{(u,\mu,\lambda,\alpha)\in X\times \R^3: u>0,\ \mu<0,\ F(u,\mu,\lambda,\alpha)=(0,0,0)\right\},
\]
thus concluding the proof of Theorem \ref{thm:main_final}.
\begin{lemma}\label{lemma:injective_case_mu<0}
Let $(u,\mu,\lambda,\alpha) \in \mathcal{S}^-$. Then the linear bounded operator
\[
F_{(u,\mu,\lambda)}(u,\mu,\lambda,\alpha):X\times\R^2 \to L^k(B_1)\times\R^2
\]
is invertible.
\end{lemma}
\begin{proof}
As in the proof of Lemma \ref{lem:nondegeneracy_in_the_ball}, it is sufficient to prove
injectivity.

As in that proof, we assume the existence of a nontrivial $(v,m,l)$ such that equations
\eqref{eq:cond_nondeg}, \eqref{eq:cond_nondeg2} hold. Since $\partial_\nu u<0$ on $\partial\Omega$,
we can test the equation for $u$ by $v^2/u \in H^1_0(\Omega)$, obtaining
\[
\begin{split}
\int_\Omega \left(\m u^{p-1}v^2 -\l v^2\right)\,dx
& =\int_\Omega \nabla u\cdot\nabla\left(\frac{v^2}{u}\right)\,dx
 =\int_\Omega \nabla u\cdot\left(2\frac{v}{u}\nabla v-\frac{v^2}{u^2}\nabla u\right)\,dx \\
& =-\int_\Omega\left|\frac{v}{u}\nabla u-\nabla v\right|^2\,dx+\int_\Omega|\nabla v|^2\,dx \\
& \leq \int_\Omega\left(p\m u^{p-1}v^2+mu^pv-luv-\l v^2\right)\,dx \\
& = \int_\Omega\left(p\m u^{p-1}v^2-\l v^2\right)\,dx.
\end{split}
\]
Therefore, being $\m<0$ and $p>1$, we must have $v\equiv0$. Finally, by testing the equation for $v$ by $u$, we deduce that $l=m\int_\Omega u^{p+1}\,dx$ and it is easy to conclude.
\end{proof}
\begin{proposition}\label{prop:defocusing_regular_curve}
$\mathcal{S}^-$ is a smooth curve, and it can be parameterized by a unique map
\[
\alpha\mapsto (u(\alpha),\mu(\a),\l(\a)), \qquad \alpha\in (\lambda_1({B_{1}}),+\infty).
\]
In particular, $u(\alpha)$ is the unique minimizer associated to $m_\alpha$ (as defined in
\eqref{eq:M_intro}).
Furthermore, $\mu'(\alpha)<0$ and $\lambda'(\alpha)<0$ for every $\alpha$.
\end{proposition}
\begin{proof}
One can use Lemma \ref{lemma:injective_case_mu<0} and reason as in the proof of Proposition \ref{prop:focusing_radial_regular_curve} in order to prove that $\mathcal{S}^-$ consists in a unique, smooth curve parameterized by $\alpha\in (\lambda_1({B_{1}}),+\infty)$, so that $u(\a)$ must achieve $m_\a$. Moreover,
all the relation contained in Corollary \ref{coro:derivate} are true also in this case.

In order to show the monotonicity of $\mu$ and $\l$, we remark that one can also prove, in a standard way, that
$u$ is the global minimizer of the related functional $J_{\mu,\lambda}$, which is bounded below and coercive since $\mu<0$.
Since $u$ is non-degenerate (by virtue of Lemma \ref{lemma:injective_case_mu<0}), we obtain that $J''_{\mu,\l}(u)[w,w]>0$ for every nontrivial $w$. But then one can reason as in the proof of Lemma \ref{lemma:lambda'<0}: using the corresponding notations, we have that in this case both $c>0$ and $b^2-ac<0$. This,
together with equation \eqref{eq:mu_u^p_v}, concludes the proof.
\end{proof}
\begin{remark}
By the above results, it is clear that $\mathcal{S}^-$ may be parameterized also with respect to $\lambda$ (or $\mu$).
Under this perspective, uniqueness and continuity for the case $p=3$ were proved in \cite{BergerFraenkel1969} (for the problem without mass constraint).
\end{remark}
We conclude by showing some asymptotic properties of $\mathcal{S}^-$ as $\alpha\to+\infty$ (the case $\alpha\to\lambda_1(\Omega)^+$ has been considered in Section \ref{section:Ambrosetti-Prodi}).
Such properties are well known in case $p=3$ since they have been studied in a different context (among others we cite \cite{BergerFraenkel1969,BethuelBrezisHelein1993,AndreShafrir1998,Serfaty2001}) and the proof can be adapted to general $p$.

\begin{proposition}
Under the notations of Proposition \ref{prop:defocusing_regular_curve}, we have that, as $\alpha\to +\infty$, $\mu\to-\infty$ and $\lambda\to -\infty$. Furthermore, if $\partial\Omega$ is smooth, then 
\[
u\to |\Omega|^{-1/2} \ \text{strongly in }L^{p+1}(\Omega), \quad \frac{\lambda}{\mu}\to |\Omega|^{-(p-1)/2} \quad \text{ and } \quad \frac{\a}{\l}\to0,
\]
as $\alpha\to+\infty$.
\end{proposition}

\begin{proof}
Since we know that $\mu$ is decreasing and that for each $\mu<0$ there exists a solution, we must have $\mu(\alpha)\to -\infty$. Moreover, $\lambda \leq -\alpha\to -\infty$.

Next we are going to show that, under the assumption $\partial\Omega$ smooth,
\begin{equation}\label{eq:last_auxiliary2}
\int_\Omega u^{p+1}\to |\Omega|^{-(p-1)/2}.
\end{equation}
To this aim notice that, by the uniqueness proved in the previous proposition, $u$ satisfies
\begin{equation*}
J_{\mu,0}(u)=\min\left\{ J_{\mu,0}(\varphi):\ \varphi\in H^1_0(\Omega),\ \int_\Omega\varphi^2\,dx=1 \right\}.
\end{equation*}
Setting, for $x\in\Omega$, $d(x):=\text{dist}(x,\partial\Omega)$, we construct a competitor function for the energy $J_{\mu,0}(u)$ as follows
\[
\varphi_\mu(x)=\left\{ \begin{array}{ll}
k^{-1}|\Omega|^{-1/2} \quad &\text{if } d(x)\geq (-\mu)^{-1/2} \\
k^{-1}|\Omega|^{-1/2} (-\mu)^{1/2} d(x) \quad &\text{if } 0\leq d(x)\leq (-\mu)^{-1/2},
\end{array}\right.
\]
where $k$ is such that $\|\varphi_\mu\|_{L^2(\Omega)}=1$. With the aid of the coarea formula, and using the fact that $\partial\Omega$ is smooth, it is possible to check that $k=1+\text{O}((-\mu)^{-1/2})$, and thus
\begin{equation}\label{eq:last_auxiliary}
\int_\Omega |\nabla \varphi_\mu|^2\, dx=\textrm{O}(\sqrt{-\mu}),\qquad \int_\Omega \left(\varphi_\mu^q-|\Omega|^{-q/2}\right)\, dx=\textrm{O}((-\mu)^{-1/2})
\end{equation}
for every $q>1$. By rewriting $J_{\mu,0}$ in the following form
\[
J_{\mu,0}(\varphi)=\int_\Omega \left\{ \frac{|\nabla\varphi|^2}{2}-
\frac{\mu}{p+1}\left(|\varphi|^{p+1}-|\Omega|^{-(p+1)/2}\right)\right\}\,dx -\frac{\mu}{p+1}|\Omega|^{-(p-1)/2},
\]
and by using the estimates \eqref{eq:last_auxiliary} with $q=p+1$, we obtain
\[
J_{\mu,0}(u) \leq J_{\mu,0}(\varphi_\mu) = \textrm{O}(\sqrt{-\mu}) -\frac{\mu}{p+1}|\Omega|^{-(p-1)/2},
\]
so that
\[
0\leq \int_\Omega \left(u^{p+1}-|\Omega|^{-(p+1)/2}\right)\,dx \leq \textrm{O}((-\mu)^{-1/2}) \to 0
\]
(by using Lemma \ref{lemma:U_not_empty_and_compact}-(iv)), so that \eqref{eq:last_auxiliary2} is proved.

Now, for each $L^2$--normalized $\varphi$, we rewrite $J_{\mu,0}(\varphi)$ as
\begin{multline*}
J_{\mu,0}(\varphi)=\int_\Omega\left\{ \frac{|\nabla \varphi|^2}{2}-\frac{\mu}{p+1}\left(|\varphi|^{(p+1)/2}-|\Omega|^{-(p+1)/4}\right)^2\right\}\, dx\\
-\frac{2\mu}{p+1}|\Omega|^{-(p+1)/4}\int_\Omega\left(|\varphi|^{(p+1)/2}-|\Omega|^{-(p+1)/4}\right)\, dx -\frac{\mu}{p+1}|\Omega|^{-(p-1)/2}.
\end{multline*}
Reasoning as before (using this time \eqref{eq:last_auxiliary} for $q=(p+1)/2$), one shows that
\begin{multline*}
\int_\Omega \left(|u|^{(p+1)/2}-|\Omega|^{-(p+1)/4}\right)^2\, dx\\
+2|\Omega|^{-(p+1)/2}\int_\Omega \left(|u|^{(p+1)/2}-|\Omega|^{-(p+1)/2}\right)\, dx\leq \textrm{O}((-\mu)^{-1/2}).
\end{multline*}
If $p\geq 3$, by H\"older inequality we have that the second integral in the l.h.s. above is nonnegative, while for $p<3$ it tends to 0 as $\alpha\to +\infty$. The latter statement is a consequence of both H\"older and interpolation inequalities, which provide
\[
\int_\Omega u^{(p+1)/2}\,dx \leq |\Omega|^{(3-p)/4}, \qquad
\|u\|_{L^{(p+1)/2}(\Omega)}\geq \|u\|_{L^{p+1}(\Omega)}^{(p-3)/(p-1)},
\]
as well as of \eqref{eq:last_auxiliary2}. Thus we have concluded that
$$
u^{(p+1)/2}\to |\Omega|^{-(p+1)/4} \qquad \text{ in } L^2(\Omega).
$$
In particular, up to a subsequence, $u\to |\Omega|^{-1/2}$ a.e. and there exists $h\in L^2$ (independent of $\alpha$) so that $|u|^{(p+1)/2}\leq h$. We can now conclude by applying Lebesgue's Dominated Convergence Theorem.

To proceed with the proof notice that, from the equality $\alpha+\lambda=\mu\int_\Omega u^{p+1}\,dx$ and Lemma \ref{lemma:U_not_empty_and_compact} (iv), we deduce
\begin{equation}\label{eq:defocusing_lambda_over_mu}
\lambda \leq \mu |\Omega|^{-(p-1)/2}.
\end{equation}
On the other hand, we have
\[
-\lambda \leq (p+1)J_{\mu,0}(u)\leq (p+1)J_{\mu,0}(\varphi_\mu) \leq C(-\mu)^{1/2} -\mu |\Omega|^{-(p-1)/2}.
\]
Dividing the last inequality by $-\mu$ and letting $\mu\to -\infty$ we obtain
\[
\limsup \frac{\lambda}{\mu} \leq |\Omega|^{-(p-1)/2},
\]
which together with \eqref{eq:defocusing_lambda_over_mu} provides the convergence of $\l/\mu$.

The last part of the statement is obtained by combining the previous asymptotics with the identity $\a/\mu=-\l/\mu+\int_\Omega u^{p+1}\,dx$.
\end{proof}

\section*{Acknowledgments} B. Noris and G. Verzini are partially supported by the PRIN2009 grant ``Critical Point Theory and Perturbative Methods for Nonlinear Differential Equations''. H. Tavares is partially supported by Funda\c c\~ao para a Ci\^encia e a Tecnologia, PEst-OE/MAT/UI0209/2013.


\noindent\verb"benedettanoris@gmail.com"\\
INdAM-COFUND Marie Curie Fellow \\
\noindent Laboratoire de Math\'ematiques, Universit\'e de Versailles-St Quentin, 45 avenue des \'Etats-Unis, 78035 Versailles cedex (France) \\

\noindent \verb"htavares@ptmat.fc.ul.pt"\\
CMAF/UL and FCUL, Av. Prof. Gama Pinto 2, 1649-003 Lisboa, Portugal \\

\noindent \verb"gianmaria.verzini@polimi.it"\\
Dipartimento di Matematica, Politecnico di Milano, p.za Leonardo da
Vinci 32,  20133 Milano, Italy

\end{document}